\documentclass[12pt]{amsart}
\usepackage[nobysame,abbrev,alphabetic]{amsrefs}
\usepackage{amssymb}
\usepackage[all]{xy}

\DeclareMathOperator{\Char}{char}

\DeclareMathOperator{\Hom}{Hom}

\DeclareMathOperator{\Img}{Im}

\DeclareMathOperator{\Infil}{Infil}

\DeclareMathOperator{\Sh}{Sh}

\DeclareMathOperator{\trg}{trg}

\DeclareFontFamily{U}{wncy}{}
\DeclareFontShape{U}{wncy}{m}{n}{<->wncyr10}{}
\DeclareSymbolFont{mcy}{U}{wncy}{m}{n}
\DeclareMathSymbol{\Sha}{\mathord}{mcy}{"58}
\DeclareMathSymbol{\sha}{\mathord}{mcy}{"78}

\begin{document}

\newtheorem{thm}{Theorem}[section]
\newtheorem{cor}[thm]{Corollary}
\newtheorem{lem}[thm]{Lemma}
\newtheorem{prop}[thm]{Proposition}
\newtheorem{defin}[thm]{Definition}
\newtheorem{exam}[thm]{Example}
\newtheorem{examples}[thm]{Examples}
\newtheorem{rem}[thm]{Remark}
\newtheorem{case}{\sl Case}
\newtheorem{claim}{Claim}
\newtheorem{prt}{Part}
\newtheorem*{mainthm}{Main Theorem}
\newtheorem*{thmA}{Theorem A}
\newtheorem*{thmB}{Theorem B}
\newtheorem*{thmC}{Theorem C}
\newtheorem*{thmD}{Theorem D}
\newtheorem{question}[thm]{Question}
\newtheorem*{notation}{Notation}
\swapnumbers
\newtheorem{rems}[thm]{Remarks}
\newtheorem*{acknowledgment}{Acknowledgment}
\newtheorem{questions}[thm]{Questions}
\numberwithin{equation}{section}

\newcommand{\ab}{\mathrm{ab}}
\newcommand{\Ad}{\mathrm{ad}}
\newcommand{\Bock}{\mathrm{Bock}}
\newcommand{\dec}{\mathrm{dec}}
\newcommand{\diam}{\mathrm{diam}}
\newcommand{\dirlim}{\varinjlim}
\newcommand{\discup}{\ \ensuremath{\mathaccent\cdot\cup}}
\newcommand{\divis}{\mathrm{div}}
\newcommand{\et}{\mathrm{et}}
\newcommand{\gr}{\mathrm{gr}}
\newcommand{\nek}{,\ldots,}
\newcommand{\ind}{\hbox{ind}}
\newcommand{\indec}{{\mathrm{indec}}}
\newcommand{\Ind}{\mathrm{Ind}}
\newcommand{\inv}{^{-1}}
\newcommand{\isom}{\cong}
\newcommand{\Massey}{\mathrm{Massey}}
\newcommand{\ndiv}{\hbox{$\,\not|\,$}}
\newcommand{\nil}{\mathrm{nil}}
\newcommand{\pr}{\mathrm{pr}}
\newcommand{\sep}{\mathrm{sep}}
\newcommand{\sh}{\mathrm{sh}}
\newcommand{\tagg}{^{''}}
\newcommand{\tensor}{\otimes}
\newcommand{\alp}{\alpha}
\newcommand{\gam}{\gamma}
\newcommand{\Gam}{\Gamma}
\newcommand{\del}{\delta}
\newcommand{\Del}{\Delta}
\newcommand{\eps}{\epsilon}
\newcommand{\lam}{\lambda}
\newcommand{\Lam}{\Lambda}
\newcommand{\sig}{\sigma}
\newcommand{\Sig}{\Sigma}
\newcommand{\bfA}{\mathbf{A}}
\newcommand{\bfB}{\mathbf{B}}
\newcommand{\bfC}{\mathbf{C}}
\newcommand{\bfF}{\mathbf{F}}
\newcommand{\bfP}{\mathbf{P}}
\newcommand{\bfQ}{\mathbf{Q}}
\newcommand{\bfR}{\mathbf{R}}
\newcommand{\bfS}{\mathbf{S}}
\newcommand{\bfT}{\mathbf{T}}
\newcommand{\bfZ}{\mathbf{Z}}
\newcommand{\dbA}{\mathbb{A}}
\newcommand{\dbC}{\mathbb{C}}
\newcommand{\dbF}{\mathbb{F}}
\newcommand{\dbK}{\mathbb{K}}
\newcommand{\dbN}{\mathbb{N}}
\newcommand{\dbP}{\mathbb{P}}
\newcommand{\dbQ}{\mathbb{Q}}
\newcommand{\dbR}{\mathbb{R}}
\newcommand{\dbU}{\mathbb{U}}
\newcommand{\dbZ}{\mathbb{Z}}
\newcommand{\grf}{\mathfrak{f}}
\newcommand{\gra}{\mathfrak{a}}
\newcommand{\grA}{\mathfrak{A}}
\newcommand{\grB}{\mathfrak{B}}
\newcommand{\grd}{\mathfrak{d}}
\newcommand{\grh}{\mathfrak{h}}
\newcommand{\grH}{\mathfrak{H}}
\newcommand{\grI}{\mathfrak{I}}
\newcommand{\grL}{\mathfrak{L}}
\newcommand{\grm}{\mathfrak{m}}
\newcommand{\grp}{\mathfrak{p}}
\newcommand{\grq}{\mathfrak{q}}
\newcommand{\grr}{\mathfrak{r}}
\newcommand{\grR}{\mathfrak{R}}
\newcommand{\calA}{\mathcal{A}}
\newcommand{\calB}{\mathcal{B}}
\newcommand{\calC}{\mathcal{C}}
\newcommand{\calE}{\mathcal{E}}
\newcommand{\calG}{\mathcal{G}}
\newcommand{\calH}{\mathcal{H}}
\newcommand{\calK}{\mathcal{K}}
\newcommand{\calL}{\mathcal{L}}
\newcommand{\calM}{\mathcal{M}}
\newcommand{\calR}{\mathcal{R}}
\newcommand{\calU}{\mathcal{U}}
\newcommand{\calW}{\mathcal{W}}
\newcommand{\calV}{\mathcal{V}}

\title[The $p$-Zassenhaus filtration and shuffle relations]
{The $p$-Zassenhaus Filtration of a Free Profinite Group and Shuffle Relations}

\author{Ido Efrat}
\address{Department of Mathematics\\
Ben-Gurion University of the Negev\\
P.O.\ Box 653, Be'er-Sheva 8410501\\
Israel} \email{efrat@bgu.ac.il}


\maketitle
\begin{abstract}
For a prime number $p$ and a free profinite group $S$ on the basis $X$, let $S_{(n,p)}$, $n=1,2\nek$ be the  $p$-Zassenhaus filtration of $S$.
For $p>n$, we give a word-combinatorial description of the cohomology group $H^2(S/S_{(n,p)},\dbZ/p)$ in terms of the shuffle algebra on $X$.
We give a natural linear basis for this cohomology group, which is constructed by means of unitriangular representations arising from Lyndon words.
\end{abstract}

\section{Introduction}
The purpose of this paper is to study the $p$-Zassenhaus filtration of a free profinite group $S$ and its cohomology by means of the combinatorics of words.
Here $p$ is a fixed prime number, and we recall that the $p$-Zassenhaus filtration of a profinite group $G$ is given by
$G_{(n,p)}=\prod_{ip^j\geq n}(G^{(i)})^{p^j}$, $n=1,2\nek$  i.e., $G_{(n,p)}$ is generated as a profinite group by all $p^j$-powers of elements of the $i$-th term of the (profinite) lower central filtration $G^{(i)}$ of $G$ for $ip^j\geq n$.

This filtration was introduced by Zassenhaus \cite{Zassenhaus39} for discrete groups (under the name \textsl{dimension subgroups modulo $p$})  as a tool to study free Lie algebras in characteristic $p$.
It proved itself to be a powerful tool in a variety of group-theoretic and arithmetic problems:
the Golod--Shafarevich solution to the class field tower problem (\cite{Koch69}, \cite{Koch02}*{\S7.7}, \cite{Zelmanov00}, \cite{Ershov12}), the structure of finitely generated pro-$p$ groups of finite rank \cite{DixonDuSautoyMannSegal99}*{Ch.\ 11},  mild groups \cite{Labute06} and one-relator pro-$p$ groups \cite{Gartner11}*{\S2.4}, multiple residue symbols and their knot theory analogs (\cite{Morishita04}, \cite{Morishita12}*{Ch.\ 8},  \cite{Vogel05}), and more.

In the Galois theory context, where $G=G_F$ is the absolute Galois group of a field $F$ containing a root of unity of order $p$, it was shown in \cite{EfratMinac17} that the quotient $G/G_{(3,p)}$ determines the full cohomology ring $H^*(G)=\bigoplus_{i\geq0}H^i(G)$ with the cup product.
Here and in the sequel we abbreviate $H^i(G)=H^i(G,\dbZ/p)$ for the profinite cohomology group of $G$ with its trivial action on $\dbZ/p$.
Moreover, $G/G_{(3,p)}$ is the smallest Galois group of $F$ with this property (see also \cite{CheboluEfratMinac12}).

In the present paper we focus on the cohomology group $H^2(G/G_{(n,p)})$ for a profinite group $G$ and $n\geq2$.
Its importance is that it controls the relator structure in the pro-$p$ group $G/G_{(n,p)}$, whereas its generators are captured by the group $H^1(G/G_{(n,p)})$, which is well understood  \cite{NeukirchSchmidtWingberg}*{\S3.9}.

Our main result gives, for a free profinite group $S$ on a basis $X$, an explicit description of  $H^2(S/S_{(n,p)})$ in terms of the \textsl{combinatorics of words}.
Namely, we consider $X$ as an alphabet with a fixed total order, and let $X^*$ be the monoid of words in $X$.
For every $n\geq0$ let $\dbZ\langle X\rangle_n$ be the free $\dbZ$-module generated by all words in $X^*$ of length $n$.
Let $\Sh(X)_{\indec,n}$ be its quotient by the submodule generated by all \textsl{shuffle products}  $u\sha v$, where $u,v$ are nonempty words in $X^*$ with $|u|+|v|=n$.
We recall for words $u=(x_1\cdots x_r)$ and $v=(x_{r+1}\cdots x_{r+s})$ in $X^*$ one defines
\[
u\sha v=\sum_\sig(x_{\sig\inv(1)}\cdots x_{\sig\inv(r+t)})\in\dbZ\langle X\rangle,
\]
where $\sig$ ranges over all permutations of $\{1,2\nek r+s\}$ such that $\sig(1)<\cdots<\sig(r)$ and $\sig(r+1)<\cdots< \sig(r+t)$.
Thus $\Sh(X)_{\indec,n}$ is the $n$-th homogenous component of the indecomposable quotient of the \textsl{shuffle algebra} $\Sh(X)$ in the sense of \cite{Efrat20}*{\S5} (see \S\ref{section on shuffle relations}).
We prove the following word-combinatorial description of $H^2(S/S_{(n,p)})$ for $p$ sufficiently large:

\begin{mainthm}
Suppose that $n<p$.
There is a canonical isomorphism of $\dbF_p$-linear spaces
\[
\label{main isomorphism}
\Bigl(\bigoplus_{x\in X}\dbZ/p\Bigr)\oplus\Bigl((\Sh(X)_{\indec,n}\tensor(\dbZ/p)\Bigr)\xrightarrow{\sim} H^2(S/S_{(n,p)}).
\]
\end{mainthm}
When $p\leq n$ we give a similar result, in the form of a canonical epimorphism.

More specifically, to any word $w$ in $X^*$ of length $1\leq |w|\leq n$ we associate a canonical cohomology element $\alp_{w,n}\in H^2(S/S_{(n,p)})$.
Then the isomorphism in the Main Theorem is induced by the map $w\mapsto \alp_{w,n}$, where $w$ is either a single-letter word, or a word of length $n$.
In these cases $\alp_{w,n}$ turns out to be a \textsl{Bockstein element}, or  an element of an \textsl{$n$-fold Massey product}, respectively (see Examples \ref{alpha for 1}--\ref{alpha for n} and the remarks below).
The construction of $\alp_{w,n}$ is based on a representation of $S/S_{(n,p)}$ in a group of \textsl{unitriangular} (i.e.,  unipotent upper-triangular) matrices, which we derive from the \textsl{Magnus map} -- see \S\ref{section on the fundamental matrix} and \S\ref{section on cohomology} for details.

A main ingredient of the proof, of independent importance, is the construction of a canonical $\dbF_p$-linear basis of $H^2(S/S_{(n,p)})$, which we call the \textsl{Lyndon basis}.
Recall that a nonempty word $w$ in $X^*$ is called a \textsl{Lyndon word} if it is smaller in the alphabetic order (induced by the fixed total order on $X$) than all its non-trivial right factors (i.e., suffixes).
The Lyndon basis then consists of all cohomology elements $\alp_{w,n}$, where $w$ is a Lyndon word of length $\lceil n/p^k\rceil$ for some $k\geq0$.
When $n\leq p$ the possible lengths are only $1$ and $n$, leading to the two direct summands in the left-hand side of the Main Theorem.

We further use Lyndon words to give a canonical  basis of the $\dbF_p$-linear space $S_{(n,p)}/(S_{(n,p)})^p[S,S_{(n,p)}]$, and prove a duality (in a unitriangular sense) between the Lyndon basis of $H^2(S/S_{(n,p)})$ and  this latter basis (Corollary \ref{unitriangular duality}).
In the smallest case $n=2$, this recovers classical duality results between Bockstein elements/cup products and $p$-powers/commutators, respectively, proved by Labute in his classical work on Demu\v skin groups (\cite{Labute66}*{Prop.\ 8}, \cite{SerreDemuskin}, \cite{NeukirchSchmidtWingberg}*{Ch.\ III, \S9}).
In the case $n=3$ it refines results by Vogel \cite{Vogel05}*{\S2}.

The paper builds upon our earlier work \cite{Efrat17}, supplemented by \cite{Efrat20}, where we proved analogous results for the \textsl{lower $p$-central filtration}, defined inductively by $G^{(1,p)}=G$ and $G^{(n,p)}=(G^{(n-1,p)})^p[G,G^{(n-1,p)}]$ for $n\geq2$.
In many respects, this filtration and the $p$-Zassenhaus filtration are the opposite extremes among the filtrations related to mod-$p$ cohomology.

While we follow the general philosophy of  \cite{Efrat17} and \cite{Efrat20}, their methods fail short to apply to the $p$-Zassenhaus filtration.
Therefore we modify these methods in several aspects:
Mainly, whereas in the lower $p$-central case one should considers word $w$ of arbitrary lengths, in the Zassenhaus filtration case we need to restrict to words of lengths $\lceil n/p^k\rceil$, $k\geq0$, as above.
These ``jumps" arise when we analyze  the filtration for the group $\dbU_i(\dbZ/p^j)$ of unitriangular $(i+1)\times(i+1)$-matrices over $\dbZ/p^j$.
They turn out to have crucial, and quite non-obvious, properties, which are in particular needed for handling the dual $S_{(n,p)}/(S_{(n,p)})^p[S,S_{(n,p)}]$ of $H^2(S/S_{(n,p)})$.
Here commutator identities due to Shalev \cite{Shalev90} also play a key role.
By contrast,  the corresponding quotient in the lower $p$-central case is $S^{(n,p)}/S^{(n+1,p)}$, which is considerably more tractable.
In addition,  the analysis for the lower $p$-central filtration in \cite{Efrat17} is based on (mixed) Lie algebra computations.
In the Zassenhaus filtration case we apply instead the theory of free \textsl{$p$-restricted} Lie algebra, following \cite{Lazard65} and \cite{Gartner11}.

The correspondence in the Main Theorem demonstrates deep conections between the $p$-Zassenhaus filtration and its cohomology and the $n$-fold Massey product  $H^1(G)^n\to H^2(G)$.
In fact, it was shown in \cite{Efrat14} that when $S$ is a free profinite group, $S_{(n,p)}/S_{(n+1,p)}$ is dual to the subgroup of $H^2(S/S_{(n,p)})$ generated by all such products.
Moreover, the latter subgroup is the kernel of the inflation map $H^2(S/S_{(n,p)})\to H^2(S/S_{(n+1,p)})$.
The size of $S_{(n,p)}/S_{(n+1,p)}$ was computed in  \cite{MinacRogelstadTan16}.
The behavior of Massey products for  absolute Galois groups $G=G_F$ has been the focus of extensive research in recent years, where the $p$-Zassenhaus filtration played an important role; see e.g.\ \cite{HopkinsWickelgren15}, \cite{MinacTan15},  \cite{EfratMatzri17}, \cite{MinacTan16}, \cite{GuillotMinacTopazWittenberg18},  \cite{HarpazWittenberg19} as well as the references therein.

\section{Hall sets}
\label{section on Hall sets}
Let $X$ be a nonempty set, considered as an alphabet.
Let again $X^*$ be the free monoid on $X$.
We consider its elements as \textsl{associative} words.
It is equipped with the binary operation $(u,v)\mapsto uv$ of associative concatenation.
Let $\calM_X$ be the \textsl{free magma} on $X$ (see \cite{SerreLie}*{Part I, Ch.\ IV, \S1}, \cite{Efrat17}*{\S2}).
Thus the elements of $\calM_X$ are the nonempty \textsl{non-associative} words in the alphabet $X$, and it is equipped with the binary operation $(u,v)\mapsto (uv)$ of non-associative concatenation.
There is a natural \textsl{foliage} (brackets dropping) map $f\colon \calM_X\to X^*$ which is the identity on $X$ (considered as a subset of both $\calM_X$ and $X^*$) and which commutes with the concatenation maps.

We fix a total order on $X$.
It induces on $X^*$ the \textsl{alphabetic} order $\leq_{\rm alp}$, which is also total.
We denote the length of a word $w\in X^*$ by $|w|$.

Let $\calH$ be a subset of words in $\calM_X$ and $\leq$ any total order on $\calH$.
We say that $(\calH,\leq)$ is a \textsl{Hall set in $\calM_X$}, if the following conditions hold \cite{Reutenauer93}*{\S4.1}:
\begin{enumerate}
\item[(1)]
$X\subseteq \calH$ as ordered sets;
\item[(2)]
If $h=(h'h'')\in\calH\setminus X$, then $h''\in\calH$ and $h<h''$;
\item[(3)]
For $h=(h'h'')\in \calM_X \setminus X$, one has $h\in\calH$ if and only if
\begin{itemize}
\item $h',h''\in\calH$ and $h'<h''$; and
\item either $h'\in X$, or $h'=(h_1h_2)$ with $h_2\geq h''$.
\end{itemize}
\end{enumerate}
In this case we say that $H=f(\calH)$ is a \textsl{Hall set in $X^*$}.

Every $w\in H$ can be written as $w=f(h)$ for a \textsl{unique} $h\in\calH$ \cite{Reutenauer93}*{Cor.\ 4.5}.
If $w\in H\setminus X$,  then we can uniquely write $h=(h'h'')$ with $h',h''\in \calH$ \cite{Reutenauer93}*{p.\ 89}.
Setting $w'=f(h'),w''=f(h'')\in H$, we call $w=w'w''$ the \textsl{standard factorization} of $w$.

\begin{exam}
\rm
The set of all Lyndon words in $X^*$ (see the Introduction) is a Hall set with respect to $\leq_{\rm alp}$ \cite{Reutenauer93}*{Th.\ 5.1}.
\end{exam}

The standard factorization of Lyndon words is explicitly given as follows:

\begin{lem}
Let $w,u,v\in X^*$ be nonempty words such that $w=uv$ and $w$ is Lyndon.
The following conditions are equivalent:
\begin{enumerate}
\item[(a)]
$w=uv$ is the the standard factorization of $w$ in the set of  Lyndon words.
\item[(b)]
$v$ is the $\leq_{\rm alp}$-minimal nontrivial right factor of $w$ which is Lyndon.
\item[(c)]
$v$ is the longest nontrival right factor of $w$ which is Lyndon.
\end{enumerate}
\end{lem}
\begin{proof}
(a)$\Leftrightarrow$(b): \quad
This is shown in the proof of \cite{Reutenauer93}*{Th.\ 5.1}.

\medskip

(b)$\Rightarrow$(c): \quad
Let $v'$ by a nontrivial Lyndon right factor of $w$.
By (b), $v\leq_{\rm alph} v'$.
Since $v'$ is Lyndon, $v$ cannot be a nontrivial right factor of $v'$.
Hence $v'$ is a right factor of $v$, so $|v'|\leq|v|$.

\medskip

(c)$\Rightarrow$(b): \quad
Let $v'$ by a nontrivial Lyndon right factor of $w$.
By (c), it is a right factor of $v$.
Since $v$ is Lyndon, $v\leq_{\rm alp}v'$.
\end{proof}

We order $\dbZ_{\geq0}\times X^*$ lexicographically with respect to the usual order on $\dbZ_{\geq0}$ and  $\leq_{\rm alp}$.
We then define a second total order  $\preceq$ on $X^*$ by setting
\begin{equation}
\label{preceq}
w_1\preceq w_2\quad \Longleftrightarrow\quad (|w_1|,w_1)\leq(|w_2|,w_2)
\end{equation}
with respect to the latter order on $\dbZ_{\geq0}\times X^*$.

\section{Lie algebras}
\label{section on Lie algebras}
Let $R$ be a unital commutative ring.
We write $R\langle X\rangle$ for the free associative $R$-algebra over the set $X$.
We view its elements as polynomials in the set $X$ of non-commuting variables and with coefficients in $R$.
Alternatively, it is the free $R$-module on the basis $X^*$ with multiplication induced by concatenation.
The algebra $R\langle X\rangle$ is graded with respect to total degree.

We write $R\langle\langle X\rangle\rangle$ for the $R$-algebra of formal power series in the set $X$ of non-commuting variables and with coefficients in $R$.

Let $k$ be a field.
For an associative $k$-algebra $A$, let $A_{\rm Lie}$ be the Lie algebra on $A$ with Lie bracket $[a,b]=ab-ba$.

We now assume that $X$ is  a nonempty totally ordered set, and fix a Hall set $H$ in $X^*$.

Let $L(X)$ be the free Lie $k$-algebra on the set $X$.
The universal enveloping algebra of $L(X)$ is $k\langle X\rangle$ \cite{SerreLie}*{Part I, Ch.\ IV, Th.\ 4.2}.

Let $L$ be a Lie $k$-algebra containing $X$.
Define a map $P_L=P_L^H\colon H\to L$ by $P_L(x)=x$ for $x\in X$, and $P_L(w)=[P_L(u),P_L(v)]$, if $w=uv$ is the standard factorization of $w$, as in \S\ref{section on Hall sets}.
This construction is functorial in $L$ in the natural sense.

\begin{prop}
\label{image of pi}
\begin{enumerate}
\item[(a)]
When $L=L(X)$,  the images $P_L(w)$, where $w\in H$, form a $k$-linear basis of  $L(X)$.
\item[(b)]
Let $L$ be a Lie $k$-algebra containing $X$.
Then the image of $P_L$ $k$-linearly spans the Lie $k$-subalgebra of $L$ generated by $X$.
\end{enumerate}
\end{prop}
\begin{proof}
(a) \quad
See \cite{Reutenauer93}*{Th.\ 4.9(i)}.

\medskip

(b) \quad
This follows from (a), the universal property of $L(X)$, and the functoriality of $P_L$.
\end{proof}

By Proposition \ref{image of pi}(a) and the Poincar\'e--Birkhoff--Witt theorem \cite{SerreLie}*{Part I, Ch.\ III, \S4}, the products $\prod_{i=1}^mP_{L(X)}(w_i)$, with $w_1\geq_{\rm alp}\cdots\geq_{\rm alp}w_m$ in $H$, form a $k$-linear basis of the universal enveloping algebra $k\langle X\rangle$ of $L(X)$.

Next assume that $\Char\,k=p>0$.
A \textsl{restricted Lie $k$-algebra} $L$ is a Lie $k$-algebra with an additional unary operation $a\mapsto a^{[p]}$ for which there is an associative $k$-algebra $A$ and a Lie $k$-algebra monomorphism $\theta\colon L\to A_{\rm Lie}$ such that $\theta(a^{[p]})=\theta(a)^p$ for every $a\in L$ (\cite{DixonDuSautoyMannSegal99}*{\S12.1}; see also \cite{Jacobson62} for an alternative equivalent definition).
A morphism of restricted Lie $k$-algebras is a morphism of Lie $k$-algebras which commutes with the $(\cdot)^{[p]}$-maps.

Every associative $k$-algebra  $A$ is endowed with the structure of a restricted Lie algebra $A_{\rm res.Lie}$, where we set $[a,b]=ab-ba$ and $a^{[p]}=a^p$.
Every restricted Lie $k$-algebra $L$ has a unique \textsl{restricted universal enveloping algebra} $\calU_{\rm res}(L)$.
This means that $\calU_{\rm res}(L)$ is an associative $k$-algebra, and the functor $A\mapsto A_{\rm res.Lie}$, from the category of  associative $k$-algebras to the category of restricted Lie $k$-algebras,  and the functor $L\mapsto \calU_{\rm res}(L)$ from the category of restricted Lie $k$-algebras to the category of associative $k$-algebras, are adjoint  (\cite{DixonDuSautoyMannSegal99}*{\S12.1}, \cite{Jacobson62}*{Ch.\ V, Th.\ 12}).

Given a restricted Lie $k$-algebra $L$ containing $X$, we define a map $\widehat{P}_L=\widehat{P}_L^H\colon \dbZ_{\geq0}\times H\to L$ by $\widehat{P}_L(j,w)=P_L(w)^{[p]^j}$ (where $(\cdot)^{[p]^j}$ denotes applying  $j$ times the operation $(\cdot)^{[p]}$).
In analogy with Proposition \ref{image of pi}(b) we have:

\begin{prop}
\label{image of hat pi}
The image of $\widehat{P}_L$ $k$-linearly spans the restricted Lie $k$-subalgebra of $L$ generated by $X$.
\end{prop}
\begin{proof}
Let $\widehat{L}_0$ be the $k$-linear subspace of $L$ spanned by $\Img(\widehat{P}_L)$.
Let $L_0$ be the $k$-linear subspace of $L$ spanned by $\Img(P_L)$.
Clearly, $X\subseteq L_0\subseteq \widehat{L}_0$.
By Proposition \ref{image of pi}(b), $L_0$ is the Lie $k$-subalgebra of $L$ generated by $X$.

Since $\Char\,k=p$, the binomial formula implies that the subspace $\widehat{L}_0$ is closed under $(\cdot)^{[p]}$.

If $w,u\in H$, then $[P_L(w),P_L(u)]\in L_0$.
It follows from the $k$-bilinearity of the Lie bracket that for every $\alp,\beta\in L_0$ also $[\alp,\beta]\in L_0$.
By induction on $m\geq1$, the $m$-times iterated Lie brackets
\[
[\alp,_m,\beta]=[\alp,[\alp,[\cdots [\alp,\beta]\cdots]]], \  \ [\alp,\beta,_m]=[\cdots[[\alp,\beta],\beta],\cdots,\beta]
\]
are also contained in $L_0$.
Using the identities $[\alp,\beta^{[p]}]=[\alp,_p,\beta]$ and $[\alp^{[p]},\beta]=[\alp,\beta,_p]$ (see \cite{DixonDuSautoyMannSegal99}*{p.\ 297}), we deduce that $[\alp^{[p]^j},\beta^{[p]^r}]\in L_0$ for every $j,r\geq0$.
By the bilinearity again, $\widehat{L}_0$ is therefore closed under the Lie bracket.

Hence $\widehat{L}_0$ is the restricted Lie $k$-subalgebra of $L$ generated by $X$.
\end{proof}

There is a \textsl{free restricted $k$-algebra} $\widehat{L}(X)$ on the generating set $X$, with the standard universal property.
It is the restricted Lie $k$-subalgebra of $k\langle X\rangle_{\rm res.Lie}$ generated by $X$, and its restricted universal enveloping algebra is $k\langle X\rangle$ \cite{Gartner11}*{Prop.\ 1.2.7}.
We note that in the algebra $k\langle X\rangle$ one has
\[
\widehat{P}_{\widehat{L}(X)}(j,w)=P_{L(X)}(w)^{p^j}
\]
for every $j\geq0$ and $w\in H$.
The following analog of Proposition \ref{image of pi}(a) generalizes a result of G\"artner \cite{Gartner11}*{Th.\ 1.2.11} (who considers a specific Hall family $H$).

\begin{cor}
\label{basis of Lres(X)}
The polynomials $\widehat{P}_{\widehat{L}(X)}(j,w)$, where $j\geq0$ and $w\in H$, form a $k$-linear basis of $\widehat{L}(X)$.
\end{cor}
\begin{proof}
We consider $\widehat{L}(X)$ as  a $k$-linear subspace of $k\langle X\rangle$.
By Proposition \ref{image of hat pi}, it is spanned by the powers $\widehat{P}_{\widehat{L}(X)}(j,w)$, where $j\geq0$ and $w\in H$.
As observed above, the products $\prod_{i=1}^mP_{L(X)}(w_i)$, with $w_1\geq_{\rm alp}\cdots\geq_{\rm alp}w_m$ in $H$, form a $k$-linear basis of $k\langle X\rangle$.
In particular, the powers $\widehat{P}_{\widehat{L}(X)}(j,w)=P_{L(X)}(w)^{p^j}$ are $k$-linearly independent.
Hence they form a $k$-linear basis of $L(X)_{\rm res}$.
\end{proof}

We grade $L(X)$ and $\widehat{L}(X)$ by total degree, and write $L(X)_n$, $\widehat{L}(X)_n$ for their homogenous components of degree $n$.

\begin{cor}
\label{basis for n component of Lres(X)}
Let $n$ be a positive integer.
\begin{enumerate}
\item[(a)]
The $P_{L(X)}(w)$, with $w\in H$ and $|w|=n$, form a $k$-linear basis of $L(X)_n$.
\item[(b)]
The $\widehat{P}_{\widehat{L}(X)}(j,w)$, with  $j\geq0$ and $w\in H$ satisfying $n=|w|p^j$, form a $k$-linear basis of $\widehat{L}(X)_n$.
\end{enumerate}
\end{cor}
\begin{proof}
(a) \quad
This follows from Proposition \ref{image of pi}(a), since $P_{L(X)}(w)$ has degree $|w|$ in $k\langle X\rangle$.
\medskip

(b)\quad
This follows from Corollary \ref{basis of Lres(X)}, since $\widehat{P}_{\widehat{L}(X)}(j,w)$ has degree $|w|p^j$ in $k\langle X\rangle$.
\end{proof}

\section{The $p$-Zassenhaus filtration}
\label{section on the p-Zassenhaus filtration}
We fix as before a prime number $p$.
For an integer $1\leq i \leq n$ let $j_n(i)=\lceil\log_p(n/i)\rceil$, i.e., $j_n(i)$ is the least integer $j$ such that $ip^j\geq n$.

\begin{lem}
\label{conditions for J(n)}
The following conditions on $1\leq i\leq n$ are equivalent:
\begin{enumerate}
\item[(a)]
 $i'p^{j_n(i')}\geq ip^{j_n(i)}$ for every $1\leq i'\leq i$.
\item[(b)]
$i=\lceil n/p^k\rceil$ for some $k\geq0$.
\end{enumerate}
\end{lem}
\begin{proof}
Set $i_k=\lceil n/p^k\rceil$.
Thus $i_0=n$, and the sequence $i_k$ is weakly decreasing to $1$.
We may restrict ourselves to $k$ such that $p^k\leq n$.
Then $(n/p^k)+1\leq n/p^{k-1}$, so $n/p^k\leq i_k<n/p^{k-1}$.
Thus $j_n(i_k)=k$. 

Since $n/p^k\leq\lceil n/p^{k+1}\rceil p$, one has $i_kp^k\leq i_{k+1}p^{k+1}$, i.e.,  the sequence $i_kp^{j_n(i_k)}$ is weakly increasing in the above range.

We also observe that if $i< i_{k-1}$, then $i<n/p^{k-1}$, i.e.,  $j_n(i)\geq k$.

\medskip

(a)$\Rightarrow$(b):
\quad
Since (b) certainly holds for $i=n$, we may assume that $i< n$, so there is $k$ in the above range such that $i_k\leq  i< i_{k-1}$.
By the previous observation, $j_n(i)\geq k$.
We take in (a) $i'=i_k$ to obtain that
\[
i_kp^{j_n(i_k)}\geq ip^{j_n(i)}\geq  i_kp^k=i_kp^{j_n(i_k)}.
\]
Hence $i=i_k$.

\medskip

(b)$\Rightarrow$(a): \quad
Suppose that $1\leq i'<i_k$.
There exists $l$ in the above range such that $i_l\leq i'<i_{l-1}$.
Necessarily $l>k$, so $i_lp^l\geq  i_kp^k$.
As we have observed, $j_n(i')\geq l$.
Hence
\[
i'p^{j_n(i')}\geq  i_lp^l\geq  i_kp^k=i_kp^{j_n(i_k)}.
\qedhere
\]
\end{proof}

We define $J(n)$ to be the set of all $1\leq i\leq n$ such that the equivalent conditions of Lemma \ref{conditions for J(n)} hold.

\begin{rem}
\label{J(n) for p large}
\rm
(1) \quad
When $n\leq p$ one has $J(n)=\{1,n\}$.

\medskip

(2) \quad
Let $1\neq i\in J(n)$ and take $k$ such that $i=\lceil n/p^k\rceil$.
By the first paragraph of the proof of Lemma \ref{conditions for J(n)}, $j_n(i)=k$.
\end{rem}

Now let $G$ be a profinite group.
Given closed subgroups $K,K'$ of $G$ and a positive integer $m$, we write $[K,K']$ (resp., $K^m$) for the closed subgroup of $G$ generated by all commutators $[k,k']=k\inv (k')\inv kk'$ (resp., powers $k^m$) with $k\in K$ and $k'\in K'$.

Recall that the (profinite) \textsl{lower central series} $G^{(i)}$, $i=1,2\nek$ of $G$ is defined inductively by
$G^{(1)}=G$, $G^{(i+1)}=[G,G^{(i)}]$.
As in the Introduction, we denote the \textsl{$p$-Zassenhaus filtration} of $G$ by $G_{(n,p)}$, $n=1,2,\ldots$ .
Since $G^{(i)}\leq G^{(n)}$ for $i>n$,
\[
G_{(n,p)}=\prod_{ip^j\geq n}(G^{(i)})^{p^j}=\prod_{i=1}^n(G^{(i)})^{p^{j_n(i)}}.
\]
The subgroups $G_{(n,p)}$ of $G$ are characteristic, hence normal.
We note that $G^{(n)}\leq G_{(n,p)}$.

The Zassenhaus filtration  can also be defined inductively by
\begin{equation}
\label{inductive definition}
G_{(1,p)}=G, \quad G_{(n,p)}=(G_{(\lceil n/p\rceil,p)})^p\prod_{i+j=n}[G_{(i,p)},G_{(j,p)}]
\end{equation}
for $n\geq2$.
Indeed, this follows from a theorem of Lazard in the case of discrete groups (\cite{DixonDuSautoyMannSegal99}*{Th.\ 11.2}, \cite{Lazard65}*{3.14.5}), and the profinite analog follows by a density argument.
It follows from (\ref{inductive definition}) that for $n\geq2$,
\begin{equation}
\label{np}
G_{(np,p)}\leq (G_{(n,p)})^p[G,G_{(n,p)}].
\end{equation}

Let $r\geq0$.
The following identity was proved  in the discrete case by Shalev \cite{Shalev90}*{Prop.\ 1.2};
the profinite analog follows again by a density argument:
\[
\prod_{ip^j\geq n}(G^{(i+r+1)})^{p^j}=[G,\prod_{ip^j\geq n}(G^{(i+r)})^{p^j}].
\]
In particular,
\begin{equation}
\label{Shalev Formula}
\prod_{ip^j\geq n}(G^{(i+1)})^{p^j}=[G,G_{(n,p)}].
\end{equation}

\begin{prop}
\label{minimal vanish}
Let $1\leq i\leq n$ be an integer such that $i\not\in J(n)$.
Then
\[
(G^{(i)})^{p^{j_n(i)}}\leq (G_{(n,p)})^p[G,G_{(n,p)}].
\]
\end{prop}
\begin{proof}
As $i\not\in J(n)$, there exists $1\leq i'<i$ such that $ip^{j_n(i)}>i'p^{j_n(i')}$.
We abbreviate $j=j_n(i)$ and $j'=j_n(i')$, so $ip^j,i'p^{j'}\geq n$.

If $j>j'$, then
\[
(G^{(i)})^{p^j}\leq((G^{(i')})^{p^{j'}})^{p^{j-j'}}\leq (G_{(n,p)})^p.
\]

If $j\leq j'$, then the inequality $i>i'p^{j'-j}$ and (\ref{Shalev Formula}) give
\[
(G^{(i)})^{p^j}\leq
(G^{(i'p^{j'-j}+1)})^{p^j}\leq [G,G_{(n,p)}].
\qedhere
\]
\end{proof}

It follows from (\ref{inductive definition}) that for every $n$ the quotient $G_{(n,p)}/G_{(n+1,p)}$ is abelian of exponent dividing $p$.
Consider the graded $\dbF_p$-module
\[
\gr G=\bigoplus_{n\geq0} G_{(n,p)}/G_{(n+1,p)}.
\]
The commutator map and the $p$-power map induce on $\gr G$ the structure of a $p$-restricted Lie $\dbF_p$-algebra (see \cite{DixonDuSautoyMannSegal99}*{\S12.2}, \cite{Gartner11}*{Prop.\ 1.2.14}).

We now specialize to the case where $S$ is a free profinite group on the basis $X$, in the sense of \cite{FriedJarden08}*{\S17.4}.
It is the inverse limit of the free profinite groups on finite subsets of $X$  \cite{FriedJarden08}*{Lemma 17.4.9}, so in our following results one may assume whenever convenient that $X$ is actually finite, and use limit arguments for the general case.

By \cite{Gartner11}*{Th.\ 1.3.8}, there is a well-defined isomorphism $\gr S\xrightarrow{\sim}\widehat{L}(X)$ of graded restricted Lie algebras.
Specifically, the coset of $x\in X$ in $\gr_1 S=S/S_{(2,p)}$ maps to $x$.

Let $H$ be as before a fixed Hall set in $X^*$.
For every word $w\in H$ we associate an element $\tau_w\in S$ as in \cite{Efrat17}.
Thus $\tau_{(x)}=x$ for $x\in X$, and for a word $w\in H$ of length $i>1$ with standard factorization $w=uv$, where $u,v\in H$ (see \S\ref{section on Hall sets}), we set $\tau_w=[\tau_u,\tau_v]$.
Then $\tau_w\in S^{(i)}$.
Hence, if $ip^j\geq n$, then  $\tau_w^{p^j}\in (S^{(i)})^{p^j}\leq S_{(n,p)}$.

\begin{prop}
\label{basis for n-component of gr S}
Let $n\geq1$.
The cosets of the powers $\tau_w^{p^j}$, with $w\in H$ and $n=|w|p^j$, form an $\dbF_p$-linear basis of  $S_{(n,p)}/S_{(n+1,p)}$.
\end{prop}
\begin{proof}
We use the terminology of \S\ref{section on Lie algebras} with the ground field $k=\dbF_p$.
By induction on the structure of $w$, the isomorphism $\gr S\xrightarrow{\sim}\widehat{L}(X)$ of restricted Lie $\dbF_p$-algebras maps the coset of $\tau_w$ to $P_{L(X)}(w)$.
Therefore it maps the coset of $\tau_w^{p^j}$ to $\widehat{P}_{\widehat{L}(X)}(j,w)=P_{L(X)}(w)^{p^j}$ considered as polynomials in $\dbF_p\langle X\rangle$.
The assertion now follows from Corollary \ref{basis for n component of Lres(X)}(b).
\end{proof}

\begin{rem}
\rm
Vogel \cite{Vogel05}*{Ch.\ I, \S3} uses a specific Hall set $H$ to give $\dbF_p$-linear bases of $S_{(n,p)}/S_{(n+1,p)}$ for $n=2,3$, as well as generating sets for arbitrary $n$.
Namely for similarly defined basic commutators $c_w\in S^{(i)}$ of words $w\in H$ with $|w|=i$, the generating set consists of all $c_w^{p^j}$ with $n=ip^j$.
Furthermore, according to \cite{MinacRogelstadTan16}*{Cor.\ 3.12}  the set of all such powers forms a basis of  $S_{(n,p)}/S_{(n+1,p)}$, however the proof lacks details.
I thank J.\ Min\'a\v c for a correspondence on the latter reference.
\end{rem}

For a word $w\in H$ of length $1\leq i\leq n$ we abbreviate
\[
\sig_w=\tau_w^{p^{j_n(i)}}.
\]
Thus $\sig_w\in S_{(n,p)}$.

\begin{thm}
\label{generators}
The cosets of $\sig_w$, where  $w\in H$ has length $i\in J(n)$, generate $S_{(n,p)}/(S_{(n,p)})^p[S,S_{(n,p)}]$.
\end{thm}
\begin{proof}
Proposition \ref{basis for n-component of gr S} implies by induction on $r\geq1$, that $S_{(n,p)}/S_{(n+r,p)}$ is generated by the cosets of $\tau_w^{p^j}$,
where $w\in H$ has length $i$, and $n\leq ip^j<n+r$.
We apply this for $n+r=np$.
By (\ref{np}), $S_{(np,p)}\leq (S_{(n,p)})^p[S,S_{(n,p)}]$, and we deduce that $S_{(n,p)}/(S_{(n,p)})^p[S,S_{(n,p)}]$ is generated by the cosets of $\tau_w^{p^j}$,
where $w\in H$ has length $i$ and $n\leq ip^j<np$.
Moreover, it suffices to take such powers with $j=j_n(i)$, since otherwise $\tau_w^{p^j}\in (S_{(n,p)})^p$.
Finally, by Proposition \ref{minimal vanish}, if $i\not\in J(n)$, then the coset of $\sig_w=\tau_w^{p^{j_n(i)}}$ is trivial.
We are therefore left with with the generators $\sig_w$ as in the assertion.
\end{proof}

\section{The fundamental matrix}
\label{section on the fundamental matrix}
For a profinite ring $R$, let $R\langle\langle X\rangle\rangle^\times$ be the  group of invertible elements in $R\langle\langle X\rangle\rangle$ (see \S\ref{section on Lie algebras}).
As before, let $S$ be the free profinite group over the basis $X$.
The \textsl{continuous Magnus homomorphism}
\[
\Lam=\Lam_R\colon S\to R\langle\langle X\rangle\rangle^\times
\]
is defined on the (profinite) generators $x\in X$ of $S$ by $\Lam(x)=1+x$;
See \cite{Efrat14}*{\S5} for details, and note that $1+x$ is invertible by the geometric progression formula.
For an arbitrary $\sig\in S$ we write
\[
\Lam(\sig)=\sum_{w\in X^*}\eps_{w,R}(\sig)w,
\]
with $\eps_{w,R}(\sig)\in R$.
The map $\eps_{w,R}\colon S\to R$ is continuous, and $\eps_{\emptyset,R}(\sig)=1$ for every $\sig$ (where $\emptyset$ denotes the empty word).

Let $\dbU_i(R)$ be the profinite group of all unitriangular $(i+1)\times(i+1)$-matrices over $R$.
Given a word $w=(x_1\cdots x_i)\in X^*$ of length $i$, we define a continuous map $\rho_w\colon S\to\dbU_i(R)$ by
\[
\rho_w(\sig)=(\eps_{(x_kx_{k+1}\cdots x_{l-1}),R}(\sig))_{1\leq k\leq l\leq i+1}.
\]
The fact that $\Lam$ is a homomorphism implies that $\rho_w$ is  a homomorphism of profinite groups \cite{Efrat14}*{Lemma 7.5}.
We call it the \textsl{Magnus representation} of $S$ corresponding to $w$.

The subgroup $S^{(n)}$ of $S$ is characterized in terms of the Magnus map as the set of all $\sig\in S$ such that $\eps_{w,\dbZ_p}(\sig)=0$ for every word $w$ of length $1\leq i<n$ \cite{Efrat17}*{Prop.\ 4.1(a)}.
The following result gives similar restrictions on the Magnus coefficients of elements of $S_{(n,p)}$.
In the discrete case, it was proved in \cite{ChapmanEfrat16}*{Example 4.6}, where it was further shown that these restrictions in fact characterize $S_{(n,p)}$.
While it is possible to derive the proposition from the discrete case using a density argument, we provide a direct proof.

\begin{prop}
\label{Zassenhaus filtration via Magnus coefficients}
If $\sig\in S_{(n,p)}$, then $\eps_{w,\dbZ_p}(\sig)\in p^{j_n(i)}\dbZ_p$ for every  word $w\in X^*$ of length $i\geq1$.
\end{prop}
\begin{proof}
Consider the subset
\[
I=\sum_{i\geq1}\sum_{|w|=i}p^{j_n(i)}\dbZ_pw
=\sum_{1\leq i\leq n}\sum_{|w|=i}p^{j_n(i)}\dbZ_pw+\sum_{i>n}\sum_{|w|=i}\dbZ_pw
\]
of $\dbZ_p\langle\langle X\rangle\rangle$.
It is an ideal in $\dbZ_p\langle\langle X\rangle\rangle$, and therefore $1+I$ is closed under multiplication.
Moreover, the identity $\alp\inv=1-\alp\inv(\alp-1)$ shows that $1+I$ is in fact a subgroup of $\dbZ_p\langle\langle X\rangle\rangle^\times$.

As $S_{(n,p)}=\prod_{i=1}^n(S^{(i)})^{p^{j_n(i)}}$, it therefore suffices to show that $\Lam_{\dbZ_p}(\tau^{p^{j_n(i)}})\in1+I$ for every $\tau\in S^{(i)}$ with $1\leq i\leq n$.
We abbreviate $j=j_n(i)$.
Then $\Lam_{\dbZ_p}(\tau)=1+\sum_{|w|\geq i}\eps_{w,\dbZ_p}(\tau)w$, by  \cite{Efrat17}*{Prop.\ 4.1(a)}.
For every $1\leq  l\leq p^j$ such that $il\leq n$ one has $p^{j_n(il)}|\binom{p^j}l$ \cite{ChapmanEfrat16}*{Example 3.9}.
Hence
\[
\begin{split}
\Lam_{\dbZ_p}(\tau^{p^j})&=\sum_{0\leq l\leq p^j}\binom{p^j}l\Bigl(\sum_{|w|\geq i}\eps_{w,\dbZ_p}(\tau)w\Bigr)^l
\subseteq 1+\sum_{1\leq l\leq p^j}\binom{p^j}l\bigl(\sum_{|w|\geq i}\dbZ_pw\bigr)^l\\
&\subseteq 1+\sum_{1\leq l\leq p^j, il\leq n}p^{j_n(il)}\bigl(\sum_{|w|\geq i}\dbZ_pw\bigr)^l+\sum_{1\leq l\leq p^j, il>n}\bigl(\sum_{|w|\geq i}\dbZ_pw\bigr)^l \\
&\subseteq1+I,
\end{split}
\]
as desired.
\end{proof}

As before, let $H$ be a Hall set in $X^*$.

\begin{cor}
\label{Magnus coef of tau to p}
Let $w,w'$ be nonempty words in $X^*$ of lengths $1\leq i, i'\leq n$, respectively, with $w'\in H$.
Then  $\eps_{w,\dbZ_p}(\sig_{w'})\in p^{j_n(i)}\dbZ_p$.
\end{cor}

For an integer $1\leq i\leq n$ let
\[
\pi_i\colon \dbZ_p\to\dbZ/p^{j_n(i)+1}
\]
be the natural epimorphism.
For  words $w,w'$ of lengths $i$, $i'$, respectively, with $w'\in H$, we define
\[
\langle w,w'\rangle_n=\pi_i(\eps_{w,\dbZ_p}(\sig_{w'})).
\]
By Corollary  \ref{Magnus coef of tau to p}, $\langle w,w'\rangle_n\in p^{j_n(i)}\dbZ_p/p^{j_n(i)+1}\dbZ_p$.
We identify the latter group with $\dbZ/p$, and thus view $\langle w,w'\rangle_n$ as an element of $\dbZ/p$.

Consider the (possibly infinite) \textsl{transposed} matrix
 \[
\Bigl[\langle w,w'\rangle_n\Bigr]_{w,w'}^T
 \]
over $\dbZ/p$, where $w,w'$ range over all words in $H$ of lengths in $J(n)$, and indexed with respect to the total order $\preceq$ on $X^*$ defined in \S\ref{section on Hall sets}.
We call it the \textsl{fundamental matrix of level $n$ of $H$}.

We now focus on the Hall set of Lyndon words  (see the Introduction).
We record the following  fundamental \textsl{triangularity property} of $H$
\cite{Reutenauer93}*{Th.\ 5.1}:
For every Lyndon word $w\in X^*$ one has
\begin{equation}
\label{triangularity}
\Lam_{\dbZ_p}(\tau_w)=1+w+\hbox{ a combination of words strictly larger than $w$ in $\preceq$}.
 \end{equation}

\begin{prop}
\label{unitriangular fundamental matrix}
Let $H$ be the Hall set of all Lyndon words in $X^*$.
The fundamental matrix of $H$ of level $n$ is unitriangular (i.e., unipotent and upper-triangular).
\end{prop}
\begin{proof}
Let $w$ be a Lyndon word of length $i\leq n$.
By (\ref{triangularity}),
\[
\Lam_{\dbZ_p}(\sig_w)=(1+w+\cdots)^{p^{j_n(i)}}=1+p^{j_n(i)}w+\cdots,
\]
where the remaining terms are multiples of of words strictly larger than $w$ in $\preceq$.
Therefore $\langle w,w\rangle_n=\pi_i(p^{j_n(i)})=1$ in $\dbZ/p$.

Furthermore, for Lyndon words $w\prec w'$ we get $\eps_{w,\dbZ_p}(\sig_{w'})=0$, whence $\langle w,w'\rangle_n=0$ (note that the empty word is not Lyndon).

Consequently, the matrix  $[\langle w,w'\rangle_n]_{w,w'}$ is unipotent \textsl{lower}-triangular, and therefore its transpose is unitriangular.
\end{proof}

\begin{exam}
\label{fundamental matrix for n=2}
\rm
Suppose that $n=2$.
Then $J(n)=\{1,2\}$.

The Lyndon words of length $\leq 2$ are the words $w=(x)$ and $w=(xy)$ with  $x,y\in X$, $x<y$.
Then  $\sig_w$ is $\tau_w^{p^{j_2(1)}}=x^p$ and $\tau_w^{p^{j_2(2)}}=[x,y]$, respectively.
In \cite{Efrat17}*{\S10} it is shown that the value of  $\langle w,w'\rangle$, where $w,w'$ are Lyndon words of lengths $\leq2$, respectively, is $1$ if $w=w'$, and is $0$ otherwise.
Thus the fundamental matrix of level $2$ for the Lyndon words is the identity matrix.
\end{exam}

\begin{exam}
\label{fundamental matrix for n=3}
\rm
Suppose that $n=3$.
Then $J(n)=\{1,3\}$ for $p\geq3$, and $J(3)=\{1,2,3\}$ for $p=2$.

The Lyndon words $w$ of length $3$ are of the forms
\[
(xxy), (xyy), (xyz), (xzy),
\]
where $x,y,z\in X$ and $x<y<z$.
For these words we have
\[
\sig_{(xxy)}=[x,[x,y]], \ \sig_{(xyy)}=[[x,y],y], \  \sig_{(xyz)}=[x,[y,z]], \  \sig_{(xzy)}=[[x,z],y],
\]
respectively.
We recall that $\langle w,w\rangle_3=1$ for every $w$, and $\langle w,w'\rangle_3=0$ when $w\prec w'$.
It remains to compute $\langle w,w'\rangle_3$ when $w'\prec w$.

If $|w|,|w'|\leq2$, then by Example \ref{fundamental matrix for n=2}, $\langle w,w'\rangle_3=0$.
We may therefore assume that $|w'|\leq|w|=3$.

If $w$ contains a letter which does not appear in $w'$, then $\eps_{w,\dbZ_p}(\sig_{w'})=0$, whence $\langle w,w'\rangle_3=0$.
Thus we may assume that every letter in $w$ appears in $w'$.

When $w=(xyy)$ and $w'=(xxy)$, where $x<y$,  the proof of \cite{Efrat17}*{Prop.\ 11.2} gives
\[
\langle w,w'\rangle_3=\eps_{(xyy)}([x,[x,y]])=0.
\]
Similarly, when $w=(xzy)$ and $w'=(xyz)$, where $x<y<z$, the proof of \cite{Efrat17}*{Prop.\ 11.2} gives
\[
\langle w,w'\rangle_3=\eps_{(xzy),\dbZ_p}([x,[y,z]])=-1.
\]

This covers all possible cases when $p\geq3$.
 When $p=2$ we also need to consider Lyndon words $w'=(xy)$ of length $2$, where $x<y$.
Then $w=(xxy)$ or $w=(xyy)$.
An explicit computation gives
\[
\Lam_{\dbZ_2}([x,y])=1+xy-yx+xyx-yxy-x^2y+y^2x+\cdots,
\]
where the remaining terms are of degree $\geq4$.
The square of this series has no terms $(xxy)$ and $(xyy)$,
so $\eps_{(xxy),\dbZ_2}([x,y]^2)=\eps_{(xyy),\dbZ_2}([x,y]^2)=0$.
Therefore  $\langle (xxy),(xy)\rangle_3=\langle (xyy),(xy)\rangle_3=0$.

Altogether, we have shown that
\[
\langle w,w'\rangle_3=
\begin{cases}
\ \ 1,& \hbox{if } w=w',\\
-1,&  \hbox{if } w=(xzy),\  w'=(xyz) \hbox{ where } x,y,z\in X, \ x<y<z,\\
\ \ 0,& \hbox{otherwise}.
\end{cases}
\]
In particular, the fundamental matrix need not be the identity matrix.
\end{exam}

\section{Unitriangular matrices}
\label{section on unitriangular matrices}
Let $i\geq1$ and $j\geq0$ be integers and consider the ring $R=\dbZ/p^{j+1}$.
In this section we study the $p$-Zassenhaus filtration of the group $\dbU=\dbU_i(R)$ of all unitriangular  $(i+1)\times(i+1)$-matrices over $R$, and in particular characterize the values of $i,j$ for which $\dbU_{(n,p)}\isom\dbZ/p$ (see \S\ref{section on the fundamental matrix} for the notation).

We denote the unit matrix in $\dbU$ by $I$, and write $E_{1,i+1}$ for the matrix which is $1$ at entry $(1,i+1)$, and is $0$ elsewhere.
For $i'\geq1$, the subgroup $\dbU^{(i')}$ of $\dbU$ consists of all matrices in $\dbU$ which are zero on the first $i'-1$ diagonals above the main diagonal \cite{BierHolubowski15}*{Th.\ 1.5(i)}.

We record the following fact about binomial coefficients:

\begin{lem}
\label{binomial coefficients}
Let $t,j'$ be positive integers such that $1\leq t\leq p^{j'}$.
The following conditions are equivalent:
\begin{enumerate}
\item[(a)]
$p^j|\binom{p^{j'}}l$, $l=1,2\nek t$;
\item[(b)]
$p^j|\binom{p^{j'}}l$ for $l=p^{\lfloor\log_pt\rfloor}$;
\item[(c)]
$j'\geq j+\lfloor\log_pt\rfloor$.
\end{enumerate}
\end{lem}
\begin{proof}
(a)$\Rightarrow$(b) is trivial.
For (a)$\Leftrightarrow$(c) and (b)$\Rightarrow$(c) see \cite{Efrat20}*{Prop.\ 2.2(c)} (and its proof).
\end{proof}

\begin{prop}
\label{powers of unitriangular group}
Let $1\leq i'\leq i$ and $j'\geq0$.
\begin{enumerate}
\item[(a)]
One has $(\dbU^{(i')})^{p^{j'}}=\{I\}$ if and only if $j'\geq j+1+\lfloor\log_p(i/i')\rfloor$.
\item[(b)]
One has  $(\dbU^{(i')})^{p^{j'}}=I+p^j\dbZ E_{1,i+1}$ if and only if $j'=j+\log_p(i/i')$ (in particular, $i/i'$ is a $p$-power).
\item[(c)]
One has $(\dbU^{(i')})^{p^{j'}}\leq I+p^j\dbZ E_{1,i+1}$ if and only if $j'\geq j+\log_p (i/i')$.
\end{enumerate}
\end{prop}
\begin{proof}
Let $N$ be an $(i+1)\times(i+1)$-matrix over $\dbZ/p^{j+1}$ such that $I+N\in\dbU^{(i')}$.
Then $N^l=0$ for every integer $l$ with $i/i'<l$.
Hence
\[
(I+N)^{p^{j'}}=\sum_{l=0}^{p^{j'}}\binom{p^{j'}}lN^l
=\sum_{l=0}^{\min(p^{j'},\lfloor i/i'\rfloor)}\binom{p^{j'}}lN^l.
\]
Further, if $i'|i$, then $N^{i/i'}\in\dbZ E_{1,i+1}$.

In particular, let $M$ be the $(i+1)\times(i+1)$-matrix over $\dbZ/p^{j+1}$ which is $1$ on the (first) super-diagonal, and is $0$ elsewhere.
Then the matrix $M^{i'l}$ is $1$ on the $i'l$-th diagonal above the main one, and is $0$ elsewhere.
In particular, $I+M^{i'}\in \dbU^{(i')}$.
As above,
\[
(1+M^{i'})^{p^{j'}}=\sum_{l=0}^{\min(p^{j'}, \lfloor i/i'\rfloor)}\binom{p^{j'}}lM^{i'l}.
\]
This matrix is $\binom{p^{j'}}l$ on the $i'l$-th diagonals above the main one, and is $0$ elsewhere.

(a) \quad
By the previous observations, $(\dbU^{(i')})^{p^{j'}}=\{I\}$ holds if and only if
\[
p^{j+1}|\binom{p^{j'}}l, \qquad l=1,2\nek\min(p^{j'},\lfloor i/i'\rfloor).
\]
In light of Lemma \ref{binomial coefficients}, this is equivalent to $j'\geq j+1+\min(j',\lfloor\log_p\lfloor i/i'\rfloor\rfloor)$, and it remains to note that $\lfloor\log_p\lfloor i/i'\rfloor\rfloor=\lfloor\log_p(i/i')\rfloor$.

\medskip

(b)\quad
First assume that $i=i'$.
Then $\dbU^{(i')}=I+\dbZ E_{1,i+1}$.
Hence $(\dbU^{(i')})^{p^{j'}}=I+\dbZ p^{j'}E_{1,i+1}$, and the equality $(\dbU^{(i')})^{p^{j'}}=I+\dbZ p^jE_{1,i+1}$ means that $j'=j$, as desired.

Next we assume that $i>i'$.
By the previous observations, $(\dbU^{(i')})^{p^{j'}}=I+\dbZ p^jE_{1,i+1}$ holds if and only if  the following conditions hold:
\begin{enumerate}
\item[(i)]
$i/i'$ is an integer $\leq p^{j'}$;
\item[(ii)]
$p^{j+1}|\binom{p^{j'}}l$, $l=1,2\nek (i/i')-1$;
\item[(iii)]
$p^j|\binom{p^{j'}}{i/i'}$, $p^{j+1}\ndiv\binom{p^{j'}}{i/i'}$.
\end{enumerate}
By Lemma \ref{binomial coefficients} again, (i)--(iii) mean that $i/i'$ is an integer $\leq p^{j'}$, and
\[
\begin{split}
j'&\geq j+1+\lfloor\log_p((i/i')-1)\rfloor\\
j'&\geq j+\lfloor\log_p(i/i')\rfloor\\
j'&<j+1+\lfloor\log_p(i/i'))\rfloor.
\end{split}
\]
This amounts to saying that $j'=j+\log_p(i/i')$.

\medskip

(c) follows from (a) and (b).
\end{proof}

The case $i'=1$ of Proposition \ref{powers of unitriangular group}(a) was shown by Sawin;
see \cite{Efrat20}*{Prop.\ 2.3}.

The following Corollary stands behind our definition of the sets $J(n)$.
In the case $i=n$ it was proved by Min\'a\v c, Rogelstad and T\^an \cite{MinacRogelstadTan16}*{Cor.\ 3.7}.

\begin{cor}
\label{U(n,p) is Z/p}
Suppose that $1\leq i\leq n$ and  $j=j_n(i)$.
One has $\dbU_{(n,p)}=I+p^{j_n(i)}\dbZ E_{1,i+1}$ if and only if $i\in J(n)$.
\end{cor}
\begin{proof}
Recall that $\dbU_{(n,p)}=\prod_{i'=1}^n(\dbU^{(i')})^{p^{j_n(i')}}$.

If $i'>i$, then $\dbU^{(i')}=\{I\}$, whence $(\dbU^{(i')})^{p^{j_n(i')}}=\{I\}$.

Taking in Proposition \ref{powers of unitriangular group}(b) $i'=i$ and $j'=j=j_n(i)$, we obtain that $(\dbU^{(i)})^{p^{j_n(i)}}=I+p^{j_n(i)}\dbZ E_{1,i+1}$.

Therefore, $\dbU_{(n,p)}=I+p^{j_n(i)}\dbZ E_{1,i+1}$ holds if and only if for every $1\leq i'\leq i$ one has $(\dbU^{(i')})^{p^{j_n(i')}}\leq I+p^{j_n(i)}\dbZ E_{1,i+1}$.
By Proposition \ref{powers of unitriangular group}(c), this inclusion is equivalent to $i'p^{j_n(i')}\geq ip^{j_n(i)}$.
\end{proof}

Thus, for $i\in J(n)$ and $\dbU=\dbU_i(\dbZ/p^{j_n(i)+1})$ there is a central extension
\begin{equation}
\label{central extension}
0\to\dbU_{(n,p)}(\isom\dbZ/p)\to\dbU\to\overline{\dbU}:=\dbU/\dbU_{(n,p)}\to1,
\end{equation}
where the isomorphism is the projection on the $(1,i+1)$-entry composed with the isomorphism $p^{j(i)}\dbZ/p^{j(i)+1}\dbZ\isom\dbZ/p$.

\section{The cohomology elements $\alp_{w,n}$}
\label{section on cohomology}
Let $S$ be again a free profinite group on the basis $X$, and let $n\geq2$.
Consider the transgression homomorphism $\trg\colon H^1(S_{(n,p)})^S\to H^2(S/S_{(n,p)})$ (recall that the cohomology groups are with respect to the coefficient module $\dbZ/p$ with trivial action).
It is the differential $d_2^{01}$ in the Lyndon-Hochschild-Serre spectral sequence corresponding to the closed normal subgroup $S_{(n,p)}$ of $S$ \cite{NeukirchSchmidtWingberg}*{Th.\ 2.4.3}.
It follows from the five term sequence in profinite cohomology \cite{NeukirchSchmidtWingberg}*{Prop.\ 1.6.7} and the fact that $S$ has cohomological dimension $1$, that $\trg$ is an isomorphism.

Now consider a word $w$ of length $i\in J(n)$.
Consider the ring $R_i=\dbZ/p^{j_n(i)+1}$, and set $\dbU=\dbU_i(R_i)$.
As before, let $\overline{\dbU}=\dbU/\dbU_{(n,p)}$.
By Corollary \ref{U(n,p) is Z/p}, the projection on the $(1,i+1)$-entry gives an isomorphism
\[
 \dbU_{(n,p)}\xrightarrow{\sim}p^{j_n(i)}\dbZ/p^{j_n(i)+1}\dbZ.
\]

The Magnus representation  $\rho=\rho_w\colon S\to\dbU$ induces continuous homomorphisms
\[
\bar\rho_{w}\colon S/S_{(n,p)}\to\overline{\dbU}, \quad \rho_{w}^0=\rho|_{S_{(n,p)}}\colon S_{(n,p)}\to\dbU_{(n,p)}.
\]
Let $\bar\rho_w^*\colon H^2(\overline{\dbU})\to H^2(S/S_{(n,p)})$ be the pullback of $\bar\rho_w$.

Let $\gam=\gam_{n,R_i}\in H^2(\overline{\dbU})$ correspond to the extension (\ref{central extension}) under the Schreier correspondence \cite{NeukirchSchmidtWingberg}*{Th.\ 1.2.4}.
We set
\[
\alp_{w,n}=\bar\rho_w^*(\gam)\in H^2(S/S_{(n,p)}).
\]

\begin{exam}
\label{alpha for 1}
$\alp_{w,n}$ for a word $w=(x)$ of length $1$. \quad
\rm
Let $j=j_n(1)=\lceil\log_pn\rceil$, so $\dbU=\dbU_1(\dbZ/p^{j+1})\isom\dbZ/p^{j+1}$.
As $1\in J(n)$, we have $\dbU_{(n,p)}\isom\dbZ/p$,  and the central extension (\ref{central extension}) becomes
\begin{equation}
\label{Bockstein extension}
0\to\dbZ/p\to\dbZ/p^{j+1}\to\dbZ/p^j\to0.
\end{equation}
We consider this extension as a sequence of trivial $S/S_{(n,p)}$-modules.
The \textsl{Bockstein homomorphism}
\[
\Bock_{p^j,S/S_{(n,p)}}\colon H^1(S/S_{(n,p)},\dbZ/p^j)\to H^2(S/S_{(n,p)})
\]
is the associated connecting homomorphism.

We may identify $\rho_{(x)}\colon S\to\dbU$ with  $\eps_{(x),\dbZ/p^{j+1}}\colon S\to\dbZ/p^{j+1}$,
and $\bar\rho_{(x)}\colon S/S_{(n,p)}\to\overline{\dbU}$ with  $\eps_{(x),\dbZ/p^j}\colon S/S_{(n,p)}\to\dbZ/p^j$, which are both continuous homomorphisms.
Thus $\alp_{(x),n}$ corresponds to the pullback of (\ref{Bockstein extension}) under $\eps_{(x),\dbZ/p^j}$.
By \cite{Efrat17}*{Remark 7.3},
\[
\alp_{(x),n}=\Bock_{p^j,S/S_{(n,p)}}(\eps_{(x),\dbZ/p^j}).
\]
\end{exam}

\medskip

For  Example \ref{alpha for n} below, we first recall a few facts about Massey products.
While these products are defined in the general context of differential graded algebras, in the special case of the $n$-fold Massey product $H^1(G,R)^n\to H^2(G,R)$ in profinite (or discrete) group cohomology it can be alternatively described in terms of unitriangular representions.
This was discovered by Dwyer  \cite{Dwyer75} in the discrete case, and we refer to \cite{Efrat14}*{\S8} for the profinite case, which is considered here.
We assume as above that $n\geq2$ and $R$ is a finite commutative ring on which $G$ acts trivially (see \cite{Wickelgren12} for the case of a nontrivial action).

Specifically, let $\dbU=\dbU_n(R)$ and let $\overline{\dbU}$ be again the quotient of $\dbU$ by the central subgroup $I+R E_{1,n+1}(\isom R^+)$.
The central extension
\begin{equation}
\label{central extension in Dwyers context}
0\to R^+\to\dbU\to \overline{\dbU}\to1
\end{equation}
of trivial $G$-modules corresponds to a cohomology element $\gam_R\in H^2(G,R^+)$.
Given $\psi_1\nek\psi_n\in H^1(G,R^+)$, we consider the continuous homomorphisms $\bar\rho\colon G\to\overline{\dbU}$ whose projection $\bar\rho_{k,k+1}\colon G\to R$ on the $(k,k+1)$-entry is $\psi_k$, for $k=1,2\nek n$.
As before, let $\bar\rho^*\colon H^2(\overline{\dbU},R^+)\to H^2(G,R^+)$ be the pullback of $\bar\rho$.
Then $\bar\rho^*(\gam_R)$ corresponds to the central extension
\[
0\to R^+\to\dbU\times_{\overline{\dbU}}G\to G\to1,
\]
where the fiber product is with respect to the natural projection $\dbU\to\overline{\dbU}$ and to $\bar\rho$.
The \textsl{$n$-fold Massey product} $\langle\psi_1\nek\psi_n\rangle$ is the subset of $H^2(G,R^+)$ consisting of all pullbacks $\bar\rho^*(\gam_R)$  \cite{Efrat14}*{Prop.\ 8.3}.
Thus the $n$-fold Massey product $\langle\cdot\nek\cdot\rangle\colon H^1(G,R^+)^n\to H^2(G,R^+)$ is a \textsl{multi-valued map}.
In the special case $n=2$ one has $\langle\psi_1,\psi_2\rangle=\{\psi_1\cup\psi_2\}$.

\begin{exam}
\label{alpha for n}
$\alp_{w,n}$ for a word $w$ of length $n\geq2$. \quad
\rm
Since $j_n(n)=0$ we have $R_n=\dbZ/p$, so $\dbU=\dbU_n(\dbZ/p)$.
As $n\in J(n)$, Corollary \ref{U(n,p) is Z/p} shows that $\dbU_{(n,p)}=I+\dbZ E_{1,n+1}\isom\dbZ/p$.
Thus (\ref{central extension in Dwyers context}) (for $R=\dbZ/p$) coincides with (\ref{central extension}) with $i=n$.

Now take a word $w=(x_1\cdots x_n)\in X^*$ of length $n$.
Let $\bar\rho=\bar\rho_w\colon S/S_{(n,p)}\to\overline{\dbU}$  and let $\bar\rho_{k,k+1}$ be homomorphisms as above.
By its definition as the pullback of  (\ref{central extension}),  $\alp_{w,n}$ is an element of the $n$-fold Massey product $\langle\rho_{12},\rho_{23}\nek\rho_{n,n+1}\rangle$ in $H^2(S/S_{(n,p)})$.
Note that $\bar\rho_{k,k+1}$ is given by $\bar\rho_{k,k+1}(x_l)=\del_{kl}$ for every $1\leq k,l\leq n$.
\end{exam}

\section{The Lyndon bases}
\label{section on bases}
We continue with the set-up of \S\ref{section on cohomology}.
Identifying $H^1(S_{(n,p)})=\Hom(S_{(n,p)},\dbZ/p)$, we obtain a non-degenerate bilinear map
\[
S_{(n,p)}/(S_{(n,p)})^p[S,S_{(n,p)}]\times H^1(S_{(n,p)})^S\to\dbZ/p, \quad (\bar\sig,\varphi)\mapsto\varphi(\sig)
\]
\cite{EfratMinac11}*{Cor.\ 2.2}.
It gives rise  to the bilinear \textsl{transgression pairing}
\begin{equation}
\begin{split}
\label{non-degenerate bilinear map}
(\cdot,\cdot)_n\colon S_{(n,p)}/(S_{(n,p)})^p&[S,S_{(n,p)}]\times H^2(S/S_{(n,p)})\to\dbZ/p, \\
&(\bar\sig,\alp)_n=-(\trg\inv\alp)(\sig),
\end{split}
\end{equation}
where $\bar\sig$ denotes the coset of $\sig\in S_{(n,p)}$.
It is therefore also non-degenerate.

By Proposition \ref{Zassenhaus filtration via Magnus coefficients} and Corollary \ref{U(n,p) is Z/p}, for a word $w$ of length $i\in J(n)$, there is a commutative diagram
\begin{equation}
\label{cd}
\xymatrix{
S_{(n,p)}\ar[r]^{\eps_{w,\dbZ_p}}\ar[d]_{\rho_w^0}&p^{j_n(i)}\dbZ_p \ar[d]^{\pi_i}& \\
\dbU_{(n,p)}\ar[r]^{\sim\qquad}&p^{j_n(i)}\dbZ/p^{j_n(i)+1}\dbZ,
}
\end{equation}
where as before, $\pi_i\colon \dbZ_p\to\dbZ/p^{j_n(i)+1}$ is  the natural projection, and the lower  isomorphism is the projection on the $(1,i+1)$-entry.
We deduce the following link between cohomology and the Magnus map.
As before, we identify $p^{j_n(i)}\dbZ/p^{j_n(i)+1}\dbZ$ with $\dbZ/p$.

\begin{prop}
\label{pairing and Magnus coef}
For $\sig\in S_{(n,p)}$ and a word $w\in X^*$ of length $i\in J(n)$ one has $(\bar\sig,\alp_{w,n})_n=\pi_i(\eps_{w,\dbZ_p}(\sig))$.
\end{prop}
\begin{proof}
The central extension (\ref{central extension}) gives rise to a transgression homomorphism
$\trg\colon H^1(\dbU^{(n,p)})^\dbU\to H^2(\overline{\dbU})$.
Let $\iota\colon\dbU_{(n,p)}\xrightarrow{\sim}\dbZ/p$ be the composition of the lower row in (\ref{cd}) with the isomorphism $p^{j_n(i)}\dbZ/p^{j_n(i)+1}\dbZ\isom\dbZ/p$.
By the results of \cite{Efrat17}*{\S7},
\[
\gam=-\trg(\iota).
\]
The functoriality of transgression gives a commutative square
\[
\xymatrix{
H^1(\dbU_{(n,p)})^\dbU\ar[r]^{\quad \trg}\ar[d]_{(\rho_{w}^0)^*} & H^2(\overline{\dbU})\ar[d]^{\bar\rho_{w}^*} \\
H^1(S_{(n,p)})^S\ar[r]^{\trg}& H^2(S/S_{(n,p)}).
}
\]
As $\sig\in S_{(n,p)}$, this square and  (\ref{cd}) give
\[
\begin{split}
(\bar\sig,\alp_{w,n})_n&=(\bar\sig,\bar\rho_w^*(\gam))_n
=-(\bar\sig,\bar\rho_w^*(\trg(\iota)))_n
=-(\bar\sig,\trg((\rho_w^0)^*(\iota)))_n  \\
&=((\rho_w^0)^*(\iota))(\sig)
=\iota(\rho_w^0(\sig))
=\pi_i(\eps_{w,\dbZ_p}(\sig)).
\end{split}
\]
\end{proof}

Now consider words $w,w'\in X^*$  of lengths $i,i'\in J(n)$, respectively, with $w'$ Lyndon.
We deduce from Proposition \ref{pairing and Magnus coef} that
\[
(\bar\sig_{w'},\alp_{w,n})_n=\pi_i(\eps_{w,\dbZ_p}(\sig_{w'}))=\langle w,w'\rangle_n.
\]
We can therefore restate Proposition \ref{unitriangular fundamental matrix} cohomologically:

\begin{cor}
\label{unitriangular duality}
The transposed matrix $\bigl[(\bar\sig_{w'},\alp_{w,n})_n\bigr]_{w,w'}^T$, where $w,w'$ range over all Lyndon words in $X^*$ of lengths $i,i'$, respectively, in $J(n)$, and totally ordered by $\preceq$, coincides with the fundamental matrix of level $n$ of the Lyndon words.
In particular, it is unitriangular, whence invertible.
\end{cor}

\begin{exam}
\rm
Let $n=2$.
Then $J(n)=\{1,2\}$.
For every $x\in X$ let $\eps_x\in H^1(S/S_{(2,p)})$ be the homomorphism induced by  $\eps_{(x),\dbZ/p}$.
It is $1$ on the coset of $x$, and is $0$ on the coset of any $x'\in X$, $x'\neq x$.

For a one letter word $w=(x)$ (which is always Lyndon) we have $\sig_w=\tau_w^p=x^p$ and $\alp_{w,2}=\Bock_{p,S/S_{(2,p)}}(\eps_x)$ (Example \ref{alpha for 1}).

For a 2-letter Lyndon word  $w=(xy)$, $x<y$, the projections of the representation $\bar\rho_w$ on the $(1,2)$- and $(2,3)$-entries are $\bar\rho_{12}=\eps_x$ and $\bar\rho_{23}=\eps_y$.
Thus $\sig_w=\tau_w=[x,y]$, and $\alp_{w,2}=\eps_x\cup\eps_y$ (Example \ref{alpha for n}).

Recall that the fundamental matrix for Lyndon words and for $n=2$ is the identity matrix (Example \ref{fundamental matrix for n=2}).
Thus we recover the fundamental duality, discovered by Labute,  between Bockstein elements/cup products and $p$-th powers/commutators, respectively (\cite{Labute66}*{Prop.\ 8}, \cite{Labute67}*{\S2}, \cite{NeukirchSchmidtWingberg}*{Ch.\ III, \S9}).
\end{exam}

We will need the following elementary fact in linear algebra  \cite{Efrat17}*{Lemma 8.4}:

\begin{lem}
\label{linear algebra lemma}
Let $R$ be a commutative ring and let $(\cdot,\cdot)\colon A\times B\to R$ be a non-degenerate bilinear map of $R$-modules.
Let $(I,\leq)$ be a finite totally ordered set, and for every $w\in I$ let $a_w\in A$, $b_w\in B$.
Suppose that the matrix $\bigl[(a_w,b_{w'})\bigr]_{w,w'\in I}$ is invertible, and that $a_w$, $w\in I$, generate $A$.
Then $a_w$, $w\in I$, is an $R$-linear basis of $A$, and $b_w$, $w\in I$, is an $R$-linear basis of $B$.
\end{lem}

We now deduce our first main result.
Note that part (a) of the theorem strengthens Theorem \ref{generators} in the special case where $H$ is the Hall set of Lyndon words.

\begin{thm}
\label{bases}
\begin{enumerate}
\item[(a)]
The $\dbF_p$-linear space $S_{(n,p)}/(S_{(n,p)})^p[S,S_{(n,p)}]$ has a basis consisting of
the cosets $\bar\sig_w$ of $\sig_w$, where $w$ is a Lyndon word in $X^*$ of length $i\in J(n)$.
\item[(b)]
The $\dbF_p$-linear space $H^2(S/S_{(n,p)})$ has a basis consisting of all $\alp_{w,n}$, where $w$ is a  Lyndon word in $X^*$ of length $i\in J(n)$.
\end{enumerate}
\end{thm}
\begin{proof}
First assume that $X$ is finite.
By Theorem \ref{generators}, the cosets in (a) generate $S_{(n,p)}/(S_{(n,p)})^p[S,S_{(n,p)}]$.
Furthermore, the bilinear map $(\cdot,\cdot)_n$ of (\ref{non-degenerate bilinear map}) is non-degenerate, and the fundamental matrix  $\bigl[(\bar\sig_{w'},\alp_{w,n})_n\bigr]_{w,w'}$ is invertible, by Corollary \ref{unitriangular duality}.
Therefore Lemma \ref{linear algebra lemma} implies both assertions.

The case of general $X$ follows from the finite case by a standard limit argument (see \cite{NeukirchSchmidtWingberg}*{Prop.\ 1.2.5}).
\end{proof}

When $2\leq n\leq p$ we have $J(n)=\{1,n\}$ (Remark \ref{J(n) for p large}(1)), $j_n(1)=1$, and $j_n(n)=0$.
In view of Examples \ref{alpha for 1} and \ref{alpha for n}, we deduce:

\begin{cor}
\label{bases when p is large}
Suppose that $2\leq n\leq p$.
\begin{enumerate}
\item[(a)]
The $\dbF_p$-linear space $S_{(n,p)}/(S_{(n,p)})^p[S,S_{(n,p)}]$ has a basis consisting of:
\begin{enumerate}
\item[(i)]
the cosets of $x^p$, $x\in X$;
\item[(ii)]
the cosets of $\tau_w$, where $w$ is a Lyndon word in $X^*$ of length $n$.
\end{enumerate}
\item[(b)]
The $\dbF_p$-linear space $H^2(S/S_{(n,p)})$ has a basis consisting of:
\begin{enumerate}
\item[(i)]
The Bockstein elements $\Bock_{p,S/S_{(n,p)}}(\eps_{(x),\dbZ/p})=\alp_{(x),n}$, $x\in X$;
\item[(ii)]
The $n$-fold Massey product elements $\alp_{w,n}$, where $w$ is a Lyndon word in $X^*$ of length $n$.
\end{enumerate}
\end{enumerate}
\end{cor}

The number of words of a given length in a Hall set $H$ can be expressed in terms of  {\sl Witt's necklace function}, defined for integers $i,m\geq1$ by
\[
\varphi_i(m)=\frac1i\sum_{d|i}\mu(d)m^{i/d}.
\]
Here $\mu$ is the M\"obius function, i.e., $\mu(d)=(-1)^k$, if $d$ is a product of $k$ distinct prime numbers, and $\mu(d)=0$ otherwise.
We also set $\varphi_i(\infty)=\infty$.
Then, the number of words of length $i$ in $H$ is $\varphi_i(|X|)$ \cite{Reutenauer93}*{Cor.\ 4.14}
We deduce:

\begin{cor}
\begin{enumerate}
\item[(a)]
For every $n\geq2$ one has
\[
\dim_{\dbF_p}H^2(S/S_{(n,p)})=\sum_{i\in J(n)}\varphi_i(|X|).
\]\
\item[(b)]
If $2\leq n\leq p$, then $\dim_{\dbF_p}H^2(S/S_{(n,p)})=|X|+\varphi_n(|X|)$.
\end{enumerate}
\end{cor}

\section{Shuffle relations}
\label{section on shuffle relations}
Recall that the shuffle product $u\sha v$ of words $u$, $v$ has been defined in the Introduction.
It extends naturally to a bilinear, commutative and associative product map $\sha\colon \dbZ\langle X\rangle\times\dbZ\langle X\rangle\to\dbZ\langle X\rangle$.
The \textsl{shuffle algebra} $\Sh(X)$ on $X$ is the graded $\dbZ$-algebra whose underlying module is the free module on $X^*$ (graded by the length of words), and its multiplication is $\sha$.

We define the \textsl{infiltration product} $u\downarrow v$ of words $u=(x_1\cdots x_r)$, $v=(x_{r+1}\cdots x_{r+t})$ in $X^*$ as follows (see \cite{ChenFoxLyndon58},  \cite{Reutenauer93}*{pp.\ 134--135}).
Consider all maps $\sig\colon\{1,2\nek r+t\}\to\{1,2\nek r+t\}$ with $\sig(1)<\cdots<\sig(r)$ and $\sig(r+1)<\cdots<\sig(r+t)$, and which satisfy the following
weak form of injectivity:
If $\sig(i)=\sig(j)$, then $x_i=x_j$.
Let the image of $\sig$ consist of $l_1<\cdots<l_{m(\sig)}$.
Then we set
\begin{equation}
\label{infiltration}
u\downarrow v=\sum_\sig(x_{\sig\inv(l_1)}\cdots x_{\sig\inv(l_{m(\sig)})})\in\dbZ\langle X\rangle.
\end{equation}
By our assumption, $x_{\sig\inv(l_i)}$ does not depend on the choice of the preimages $\sig\inv(l_i)$ of $l_i$.
We also write $\Infil(u,v)$ for the set of all such words $(x_{\sig\inv(l_1)}\cdots x_{\sig\inv(l_{m(\sig)})})$.
Thus  $u\sha v$ is the part of $u\downarrow v$ of degree $r+t$, that is, the partial sum corresponding to all such maps $\sig$ which in addition are bijective.
The product $\downarrow$ on words extends by linearity to an associative and commutative bilinear map on $\dbZ\langle X\rangle$.

There is a well defined $\dbZ_p$-bilinear map
\[
(\cdot,\cdot)\colon \dbZ_p\langle\langle X\rangle\rangle\times\dbZ_p\langle X\rangle\to\dbZ_p, \quad
(f,g)=\sum_{w\in X^*}f_wg_w,
\]
where $f_w,g_w$ are the coefficients of $f$, $g$, respectively, at $w$ \cite{Reutenauer93}*{p.\ 17}.

The following connection between the Magnus representation and the infiltration product is proved in the discrete case in \cite{ChenFoxLyndon58}*{Th.\ 3.6}.
We refer to \cite{Vogel05}*{Prop.\ 2.25} and \cite{Morishita12}*{Prop.\ 8.16} for the profinite case.
Here we view the infiltration and shuffle products as elements of $\dbZ\langle X\rangle\subseteq \dbZ_p\langle X\rangle$.

\begin{prop}
\label{CFL relation}
For every $\emptyset\neq u,v\in X^*$ and every $\sig\in S$ one has
\[
\eps_{u,\dbZ_p}(\sig)\eps_{v,\dbZ_p}(\sig)=(\Lam_{\dbZ_p}(\sig),u\downarrow v).
\]
\end{prop}

\begin{cor}
\label{Lam and shuffles}
Let $u,v$ be nonempty words in $X^*$ with $i=|u|+|v|\leq n$.
For every $\sig\in S_{(n,p)}$ one has $(\Lam_{\dbZ_p}(\sig),u\sha v)\in p^{j_n(i-1)}\dbZ_p$.
\end{cor}
\begin{proof}
Let $w$ be a word of length $1\leq k\leq i-1$.
Then $j_n(k)\geq j_n(i-1)$, so by Proposition \ref{Zassenhaus filtration via Magnus coefficients},
$\eps_{w,\dbZ_p}(\sig)\in p^{j_n(k)}\dbZ_p\subseteq p^{j_n(i-1)}\dbZ_p$.
In particular, this is the case for $w=u$, $w=v$, and for $w\in\Infil(u,v)$ of length smaller than $i$.
Since $u\sha v$ is the part of $u\downarrow v$ consisting of summands of maximal length $i$, Proposition \ref{CFL relation} implies that  $(\Lam_{\dbZ_p}(\sig),u\sha v)\in p^{j_n(i-1)}\dbZ_p$.
\end{proof}

We obtain the following \textsl{shuffle relations}.
Here $X^i$ stands for the set of words in $X^*$ of length $i$.

\begin{thm}
\label{shuffle relations}
Let $\emptyset\neq  u,v\in X^*$ with $i=|u|+|v|\in J(n)$.
Then
\[
\sum_{w\in X^i}(u\sha v)_w\alp_{w,n}=0.
\]
\end{thm}
\begin{proof}
As $2\leq i\in J(n)$, we have $(i-1)p^{j_n(i-1)}\geq ip^{j_n(i)}$, whence $j_n(i-1)>j_n(i)$.

We recall that $u\sha v$ is homogenous of degree $i$.
For  $\sig\in S_{(n,p)}$, Corollary \ref{Lam and shuffles} gives
\[
\begin{split}
\sum_{w\in X^i}(u\sha v)_w\eps_{w,\dbZ_p}(\sig)&=\sum_{w\in X^*}(u\sha v)_w\eps_{w,\dbZ_p}(\sig)=(\Lam_{\dbZ_p}(\sig),u\sha v)\\
&\in p^{j_n(i-1)}\dbZ_p\subseteq p^{j_n(i)+1}\dbZ_p.
\end{split}
\]
Therefore, by Proposition \ref{pairing and Magnus coef},
\[
\begin{split}
(\bar\sig,\sum_{w\in X^i}(u\sha v)_w\alp_{w,n})_n&=\sum_{w\in X^i}(u\sha v)_w(\bar\sig,\alp_{w,n})_n
=\sum_{w\in X^i}(u\sha v)_w\,\pi_i(\eps_{w,\dbZ_p}(\sig))\\
&=\pi_i\Bigl(\sum_{w\in X^i}(u\sha v)_w\eps_{w,\dbZ_p}(\sig)\Bigr)=0.
\end{split}
\]
Now use the fact that  $(\cdot,\cdot)_n$ is non-degenerate.
\end{proof}

Given a graded $R$-algebra $A=\bigoplus_{i\geq0} A_i$,
we denote $A_+=\bigoplus_{i\geq1} A_i$.
Let $\hbox{WD}(A)$ be the $R$-submodule of $A$ generated by all products $aa'$, where $a,a'\in A_+$.
We call $\hbox{WD}(A)$ the submodule of \textsl{weakly decomposable elements} of $A$.
It is also generated by all products $aa'$, where $a,a'\in A_+$ are homogenous.
Hence the quotient $A_\indec=A/\hbox{WD}(A)$ has the structure of a graded $R$-module, which we call the \textsl{indecomposable quotient} of $A$.

Note that $\hbox{WD}(A)_0=\hbox{WD}(A)_1=\{0\}$, so the graded module morphism $A\to A_{\indec}$ is an isomorphism in degrees $0$ and $1$.
For example, when $A=R\langle X\rangle$, one has $A_{\indec,0}=R$, $A_{\indec,1}$ is the free $R$-module on the basis $X$, and $A_{\indec,i}=0$ for all $i\geq2$.

When $A=\Sh(X)$ is the shuffle algebra, we recover the module $\Sh(X)_{\indec,n}$ as defined in the Introduction.
The following key fact was proved in \cite{Efrat20}*{Prop.\ 6.3}.
It is based on a construction by Radford \cite{Radford79} and Perrin--Viennot of a basis of $\dbZ\langle X\rangle$, which arises from the decomposition of words in $X^*$ into Lyndon words.

\begin{prop}
\label{Lyndon words generate}
Suppose that $1\leq n<p$.
Then  the images of the Lyndon words of length
 $n$ span $\Sh(X)_{\indec,n}\tensor(\dbZ/p)$ as an $\dbF_p$-linear space.
\end{prop}
In fact, in \cite{Efrat20}*{Th.\ 7.3(b)} it is proved that these images form a linear basis of $\Sh(X)_{\indec,n}\tensor(\dbZ/p)$, but we shall not use this stronger result.

\begin{thm}
\label{surjectivity}
Suppose that $n\geq2$.
The map $w\mapsto\alp_{w,n}$ induces an epimorphism of $\dbF_p$-linear spaces
\[
\Bigl(\bigoplus_{i\in J(n)}\Sh(X)_{\indec,i}\Bigr)\tensor(\dbZ/p)\to H^2(S/S_{(n,p)}).
\]
\end{thm}
\begin{proof}
For $i\in J(n)$, the map $X^i\to H^2(S/S_{(n,p)})$, $w\mapsto \alp_{w,n}$, extends by linearity to a $\dbZ$-module homomorphism
\[
\Phi_i\colon \dbZ\langle X\rangle_i=\bigoplus_{w\in X^i}\dbZ w\to H^2(S/S_{(n,p)}), \quad f=\sum_{w\in X^i}f_ww\mapsto\sum_{w\in X^i}f_w\alp_{w,n}.
\]
By Theorem \ref{shuffle relations}, $\Phi_i(u\sha v)=0$ for any nonempty words $u,v\in X^*$ with $i=|u|+|v|$.
Consequently, $\Phi_i$ factors via $\Sh(X)_{\indec,i}$, and induces an $\dbF_p$-linear map
\[
\bar\Phi_i\colon \Sh(X)_{\indec,i}\tensor(\dbZ/p)\to H^2(S/S_{(n,p)}),
\]
where $\bar\Phi_i(\bar w)=\alp_{w,n}$ for $w\in X^i$.
Since the $\alp_{w,n}$, where $w$ ranges over all Lyndon words of an arbitrary length $i\in J(n)$, form an $\dbF_p$-linear basis of $H^2(S/S_{(n,p)})$ (Theorem \ref{bases}(b)), we obtain an epimorphism
\[
\label{epimorphism}
\bigoplus_{i\in J(n)}\bar\Phi_i\colon\bigl(\bigoplus_{i\in J(n)}\Sh(X)_{\indec,i}\Bigr)\tensor(\dbZ/p)\to H^2(S/S_{(n,p)}).
\qedhere
\]
\end{proof}

We now obtain the Main Theorem from the Introduction:

\begin{thm}
Suppose that $2\leq n<p$.
Then there is an isomorphism of $\dbF_p$-linear spaces
\begin{equation}
\label{main isomorphism}
\Bigl(\bigoplus_{x\in X}\dbZ/p\Bigr)\oplus\Bigl((\Sh(X)_{\indec,n}\tensor(\dbZ/p)\Bigr)\xrightarrow{\sim} H^2(S/S_{(n,p)}).
\end{equation}
Specifically, this isomorphism maps a generator $1_x$ of the $\dbZ/p$-summand at $x\in X$ to  $\Bock_{p,S/S_{(n,p)}}(\eps_{(x),\dbZ/p})$, and maps the image $\bar w$ of a word $w\in X^*$ of length $n$ to the $n$-fold Massey product element $\alp_{w,n}$.
\end{thm}
\begin{proof}
By Remark \ref{J(n) for p large}(1), $J(n)=\{1,n\}$.
Therefore Theorem \ref{surjectivity} gives an \textsl{epimorphism} as in (\ref{main isomorphism}).
The generators $1_x$ and the images $\bar w$ of words $w$ of length $n$ are mapped as specified, by Examples \ref{alpha for 1} and \ref{alpha for n}.

The generators $1_x$, $x\in X$, clearly span $\bigoplus_{x\in X}\dbZ/p$, and by Proposition \ref{Lyndon words generate}, the images $\bar w$ of the Lyndon words $w$ in $X^*$ of length $n$ span $\Sh(X)_{\indec,n}\tensor(\dbZ/p)$.
Together they form a spanning set of the left-hand side of (\ref{main isomorphism}), which is mapped to a linear basis of the right-hand side (Corollary \ref{bases when p is large}).
It follows that this spanning set is a linear basis, and the map (\ref{main isomorphism}) is an isomorphism.
\end{proof}


\begin{bibdiv}
\begin{biblist}

\bib{BierHolubowski15}{article}{
author={Bier, Agnieszka},
   author={Ho\l ubowski, Waldemar},
   title={A note on commutators in the group of infinite triangular matrices
   over a ring},
   journal={Linear Multilinear Algebra},
   volume={63},
   date={2015},
   number={11},
   pages={2301\ndash2310},
}

\bib{ChapmanEfrat16}{article}{
author={Chapman, Michael},
author={Efrat, Ido},
title={Filtrations of the free group arising from the lower central series},
journal={J.\ Group Theory},
status={a special issue in memory of O.\ Melnikov},
volume={19},
date={2016},
pages={405\ndash433},
}

\bib{CheboluEfratMinac12}{article}{
   author={Chebolu, Sunil K.},
   author={Efrat, Ido},
   author={Min{\'a}{\v{c}}, J{\'a}n},
   title={Quotients of absolute Galois groups which determine the entire Galois cohomology},
   journal={Math. Ann.},
   volume={352},
   date={2012},
   pages={205--221},
}

\bib{ChenFoxLyndon58}{article}{
   author={Chen, K.-T.},
   author={Fox, R. H.},
   author={Lyndon, R. C.},
   title={Free differential calculus. IV. The quotient groups of the lower  central series},
   journal={Ann. Math.},
   volume={68},
   date={1958},
   pages={81--95},
}

\bib{DixonDuSautoyMannSegal99}{book}{
    author={Dixon, J.D.},
	author={Du Sautoy, M.P.F.},
    author={Mann, A.},
	author={Segal, D.},
    title={Analytic Pro-$p$ Groups},
   date={1999},
   publisher={Cambridge Univ.\ Press},
   place={Cambridge},
   label={DDMS99},
   }


\bib{Dwyer75}{article}{
   author={Dwyer, William G.},
   title={Homology, Massey products and maps between groups},
   journal={J. Pure Appl. Algebra},
   volume={6},
   date={1975},
   pages={177--190},
}

\bib{Efrat14}{article}{
author={Efrat, Ido},
title={The Zassenhaus filtration, Massey products, and representations of profinite groups},
journal={Adv.\ Math.},
volume={263},
date={2014},
pages={389\ndash411},
}


\bib{Efrat17}{article}{
author={Efrat, Ido},
title={The cohomology of canonical quotients of free groups and Lyndon words},
journal={Documenta Math.},
volume={22},
date={2017},
pages={973\ndash997},
}

\bib{Efrat20}{article}{
author={Efrat,Ido},
title={The lower p-central series of a free profinite group and the shuffle algebra},
journal={J.\ Pure Appl.\ Algebra},
volume={224}
date={2020},
}



\bib{EfratMatzri17}{article}{
label={EfMa17},
author={Efrat, Ido},
author={Matzri, Eliyahu},
title={Triple Massey products and absolute Galois groups},
journal={J. Eur. Math Soc.\ (JEMS)},
volume={19},
date={2017},
pages={3629\ndash3640},
}

\bib{EfratMinac11}{article}{
label={EfMi11},
   author={Efrat, Ido},
   author={Min\'a\v c, J\'an},
   title={On the descending central sequence of absolute Galois groups},
   journal={Amer. J. Math.},
   volume={133},
   date={2011},
   pages={1503\ndash1532},
 }


\bib{EfratMinac17}{article}{
label={EfMi17},
   author={Efrat, Ido},
   author={Min\'a\v c, J\'an},
   title={Galois groups and cohomological functors},
   journal={Trans.\ Amer.\ Math.\ Soc.},
   volume={369},
   date={2017},
   pages={2697\ndash2720},
 }


\bib{Ershov12}{article}{
   author={Ershov, Mikhail},
   title={Golod-Shafarevich groups: a survey},
   journal={Internat. J. Algebra Comput.},
   volume={22},
   date={2012},
   pages={1230001, 68},
}

\bib{FriedJarden08}{book}{
   author={Fried, Michael D.},
   author={Jarden, Moshe},
   title={Field Arithmetic},
   edition={Third edition},
   publisher={Springer, Berlin},
   date={2008},
   pages={xxiv+792},
}

\bib{Gartner11}{thesis}{
author={G\"artner, Jochen},
title={Mild pro-$p$-groups with trivial cup-product},
place={Universit\"at Heidelberg},
type={Dissertation},
date={2011},
}


\bib{GuillotMinacTopazWittenberg18}{article}{
author={Guillot, Pierre},
author={Min\'a\v c, J\'an},
author={Topaz,  Adam},
author={Wittenberg, Olivier},
title={Four-fold Massey products in Galois cohomology},
journal={Compositio Math.},
volume={154},
date={2018},
pages={1921\ndash1959},
}

\bib{HarpazWittenberg19}{article}{
label={HaW19},
author={Harpaz, Yonatan},
author={Wittenberg, Olivier},
title={The Massey vanishing conjecture for number fields},
date={2019},
status={preprint},
eprint={arXiv:1904.06512},
}

\bib{HopkinsWickelgren15}{article}{
label={HoW15},
author={Hopkins, M.},
author={Wickelgren, Kirsten},
title={Splitting varieties for triple Massey products},
journal={J. Pure Appl. Algebra},
volume={219},
date={2015},
pages={1304\ndash1319},
}

\bib{Jacobson62}{book}{
author={Jacobson, Nathan},
title={Lie Groups},
publisher={Dover Publications, Inc.},
place={New York},
date={1962},
}


\bib{Koch69}{article}{
author={Koch, Helmut},
   title={Zum Satz von Golod-Schafarewitsch},
   journal={Math. Nachr.},
   volume={42},
   date={1969},
   pages={321\ndash333},
}


\bib{Koch02}{book}{
   author={Koch, Helmut},
   title={Galois theory of $p$-extensions},
   publisher={Springer, Berlin},
   date={2002},
   pages={xiv+190},
}


\bib{Labute66}{article}{
   author={Labute, John P.},
   title={Demu\v{s}kin groups of rank $\aleph _{0}$},
   journal={Bull. Soc. Math. France},
   volume={94},
   date={1966},
   pages={211\ndash244},
}


\bib{Labute67}{article}{
author={Labute, John},
title={Classification of Demu\v skin groups},
journal={Can.\ J.\ Math.},
volume={19},
date={1967},
pages={106\ndash132},
}

\bib{Labute06}{article}{
   author={Labute, John},
   title={Mild pro-$p$ groups and Galois groups of $p$-extensions of $\dbQ$},
   journal={J.\ reine angew. Math.},
   volume={596},
   date={2006},
   pages={155\ndash182},
}


\bib{Lazard65}{article}{
   author={Lazard, M.}
   title={Groupes analytiques $p$-adiques},
   journal={Inst. Hautes \'Etudes Sci., Publ. Math.},
   colume={26},
   date={1965},
   pages={389\ndash603},
}




\bib{MinacRogelstadTan16}{article}{
author={Min\'a\v c, J\'an},
author={Rogelstad, Michael},
author={T\^an, Nguyen Duy},
title={Dimensions of Zassenhaus filtration subquotients of some pro-$p$-groups},
journal={Israel J.\ Math.},
volume={212},
date={2016},
pages={825\ndash855},
}




\bib{MinacTan15}{article}{
author={Min\'a\v c, J\'an},
author={T\^an, Nguyen Duy},
title={The Kernel Unipotent Conjecture and the vanishing of Massey products for odd rigid fields {\rm (with an appendix by Efrat, I., Min\'a\v c, J.\  and T\^an, N. D.)}},
journal={Adv.\ Math.},
volume={273},
date={2015},
pages={242\ndash270},
}


\bib{MinacTan16}{article}{
author={Min\'a\v c, J\'an},
author={T\^an, Nguyen Duy},
title={Triple Massey products vanish over all fields},
journal={J.\ London Math.\ Soc.},
volume={94},
date={2016},
pages={909\ndash932},
}


\bib{Morishita04}{article}{
   author={Morishita, Masanori},
   title={Milnor invariants and Massey products for prime numbers},
   journal={Compos. Math.},
   volume={140},
   date={2004},
   pages={69--83},
}

\bib{Morishita12}{book}{
   author={Morishita, Masanori},
   title={Knots and Primes},
   series={Universitext},
   publisher={Springer, London},
   date={2012},
   pages={xii+191},
}


\bib{NeukirchSchmidtWingberg}{book}{
author={Neukirch, J.},
author={Schmidt, Alexander},
author={Wingberg, Kay},
title={Cohomology of Number Fields},
edition={Second edition},
publisher={Springer},
date={2008},
}



\bib{Radford79}{article}{
author={Radford, David E.},
title={A natural ring basis for the shuffle algebra and an application to group schemes},
journal={J.\  Algebra},
volume={58},
date={1979},
pages={432\ndash454},
}

\bib{Reutenauer93}{book}{
   author={Reutenauer, Christophe},
   title={Free Lie algebras},
   series={London Mathematical Society Monographs. New Series},
   volume={7},
   note={Oxford Science Publications},
   publisher={The Clarendon Press, Oxford University Press, New York},
   date={1993},
   pages={xviii+269},
}


\bib{SerreDemuskin}{article}{
   author={Serre, Jean-Pierre},
   title={Structure de certains pro-$p$-groupes (d'apr\`es Demu\v skin)},
   conference={
      title={S\'eminaire Bourbaki (1962/63), Exp.\ 252},
   },
   label={Ser63},
}



\bib{SerreLie}{book}{
    author={Serre, Jean-Pierre},
     title={Lie Algebras and Lie Groups},
      publisher={Springer}
      date={1992},
}

\bib{Shalev90}{article}{
author={Shalev, Aner},
title={Dimension subgroups, nilpotency indices, and the number of generators of ideals in $p$-algebras},
journal={J.\ Algebra},
volume={129},
date={1990},
pages={412\ndash438},
}


\bib{VogelThesis}{thesis}{
author={Vogel, Denis},
title={Massey products in the Galois cohomology of number fields},
type={Ph.D.\ thesis},
place={Universit\"at Heidelberg},
date={2004},
}

\bib{Vogel05}{article}{
   author={Vogel, Denis},
   title={On the Galois group of 2-extensions with restricted ramification},
   journal={J. reine angew. Math.},
   volume={581},
   date={2005},
   pages={117--150},
}
		
	
\bib{Wickelgren12}{article}{
   author={Wickelgren, Kirsten},
   title={$n$-nilpotent obstructions to $\pi_1$ sections of $\Bbb P^1-\{0,1,\infty\}$ and Massey products},
   conference={
      title={Galois-Teichm\"uller theory and arithmetic geometry},
   },
   book={
      series={Adv. Stud. Pure Math.},
      volume={63},
      publisher={Math. Soc. Japan, Tokyo},
   },
   date={2012},
   pages={579--600},
  }

\bib{Zassenhaus39}{article}{
 author={Zassenhaus, Hans},
   title={Ein Verfahren, jeder endlichen $p$-Gruppe einen Lie-Ring mit der Charakteristik $p$ zuzuordnen},
   journal={Abh. Math. Sem. Univ. Hamburg},
   volume={13},
   date={1939},
   pages={200\ndash207},
}


\bib{Zelmanov00}{article}{
   author={Zelmanov, E.},
   title={On groups satisfying the Golod-Shafarevich condition},
   conference={
      title={New horizons in pro-$p$ groups},
   },
   book={
      series={Progr. Math.},
      volume={184},
      publisher={Birkh\"{a}user Boston, Boston, MA},
   },
   date={2000},
   pages={223\ndash232},
}


\end{biblist}
\end{bibdiv}
\end{document}